\documentclass[10pt,a4paper]{article}
\usepackage[nottoc]{tocbibind}
\usepackage{graphicx}
\usepackage{slashed}
\usepackage{enumerate}
\usepackage{amsmath}
{
	
	\newtheorem{assumption}{Assumption}
}
\usepackage{amsfonts}
\usepackage{psfrag}
\usepackage{color}
\usepackage{mathtools}
\usepackage{pgf,tikz}
\usepackage{mathrsfs}
\usetikzlibrary{arrows}
\usepackage{fancyhdr}
\usepackage{amssymb}
\usepackage{slashed}
\usepackage{enumitem}
\usepackage[hidelinks]{hyperref}
\newtheorem{definition}{Definition}
\newtheorem{remark}{Remark}

\newtheorem{proposition}{Proposition}[section]

\newtheorem{theorem}{Theorem}[section]

\newtheorem{lemma}{Lemma}[section]

\numberwithin{equation}{section}
\allowdisplaybreaks[4]

\newenvironment{proof}{\smallskip\noindent\emph{Proof.}\hspace{1pt}}%
{\hspace{-5pt}{\nobreak\quad\nobreak\hfill\nobreak$\square$\vspace{8pt}%
		\par}\smallskip\goodbreak}

\newcommand{\g}{\gamma}
\newcommand{\e}{\mathrm{e}}
\newcommand{\be}{\begin{equation}}
\newcommand{\ee}{\end{equation}}

\newcommand{\Jcl}{\mathcal{J}_\text{clas}}
\newcommand{\Jrel}{\mathcal{J}_\text{rel}}

\newcommand{\w}{\tilde{w}}
\newcommand{\z}{\tilde{z}}
\newcommand{\hone}{\tilde{h}_1}
\newcommand{\htwo}{\tilde{h}_2}

\newcommand\restri[2]{{
		\left.\kern-\nulldelimiterspace 
		#1 
		\right|_{#2} 
}}

\definecolor{ffqqqq}{rgb}{1.,0.,0.}
\definecolor{uuuuuu}{rgb}{0.26666666666666666,0.26666666666666666,0.26666666666666666}

\begin{document}

	\title{Formation of singularities for the  relativistic Euler equations}
	
	\author{
		Nikolaos Athanasiou\footnote{Mathematical Institute, University of Oxford, Andrew Wiles Building, Radcliffe Observatory Quarter, Woodstock Road, Oxford, OX26GG, Oxford, United Kingdom. \hspace{1mm}E-mail: Nikolaos.Athanasiou@maths.ox.ac.uk} \\
		\texttt{Nikolaos.Athanasiou@maths.ox.ac.uk}
		\and
		Shengguo Zhu\footnote{Mathematical Institute, University of Oxford, Andrew Wiles Building, Radcliffe Observatory Quarter, Woodstock Road, Oxford, OX26GG, Oxford, United Kingdom. \hspace{1mm}E-mail: Shengguo.Zhu@maths.ox.ac.uk} \\
		\texttt{Shengguo.Zhu@maths.ox.ac.uk}
	}
	\maketitle
	
	\thispagestyle{plain}

	\vspace{-0.6cm}
	\begin{abstract}\vspace{5pt}
		
		\par \noindent 	
		This paper contributes to the study of large data problems for $C^1$ solutions of the relativistic Euler equations. In the $(1+1)$-dimensional spacetime setting, if the initial data are away from  vacuum, a key difficulty in proving the global well-posedness or  finite time blow-up  is coming up with a way to obtain sharp enough control on the lower bound of the mass-energy density function $\rho$. First,  for $C^1$ solutions  of  the 1-dimensional classical   isentropic compressible   Euler equations  in the Eulerian setting,  we show  a novel idea of obtaining a mass density  time-dependent   lower bound  by studying the difference of the two Riemann invariants, along with certain weighted gradients of them. Furthermore,  using an elaborate argument on a certain ODE inequality and introducing some key artificial (new) quantities,  we apply this idea   to obtain the lower bound estimate for the  mass-energy density of  the (1+1)-dimensional relativistic Euler equations.  Ultimately, for $C^1$ solutions with uniformly positive initial mass-energy density   of the (1+1)-dimensional relativistic Euler equations, we give a necessary and sufficient condition  for  the formation of singularity     in finite time, which gives a complete  picture for the ($C^1$)   large data problem in dimension $(1+1)$. Moreover,  for the (3+1)-dimensional relativistic fluids, under the assumption that the initial mass-energy density  vanishes in some open domain, we give two sufficient conditions for  $C^1$ solutions to blow up in finite time, no matter how small the initial data are. We also do some interesting studies on the asymptotic behavior of the  relativistic velocity,   which  tells us that  one can not obtain any global regular solution whose $L^\infty$ norm of $u$ decays to zero as time $t$ goes to infinity. 
	\end{abstract}
	\vspace{0.2cm}

	\setlength{\voffset}{-0in} \setlength{\textheight}{0.9\textheight}
	
	\setcounter{page}{1} \setcounter{equation}{0}

	\maketitle
	
	2010\textit{\ Mathematical Subject Classification:} 83A05, 	76N15,  35L65, 35L67.
	

	\section{Introduction}
	

	\vspace{3mm}
	
	\quad This paper is devoted to the Cauchy problem of the	 relativistic Euler equations (henceforth denoted by \textit{RE}) with   large data.  In fluid mechanics and astrophysics, the relativistic Euler equations are a generalization of the Euler equations that account for the effects of special relativity. On a fixed $(d+1)$-dimensional Minkowski background, they are given by 
	\begin{equation}\label{RE}
	\begin{cases}
	\displaystyle
	\big( \frac{\rho + P \lvert u\rvert^2/c^4}{1-\lvert u \rvert^2/c^2} \big)_t  + \text{div}\big( \frac{(\rho+ P/c^2)u}{1-\lvert u\rvert^2/c^2} \big) = 0, \\[12pt]
	\displaystyle
	\big( \frac{(\rho+ P/c^2)u}{1-\lvert u\rvert^2/c^2} \big)_t +  \text{div}\big( \frac{(\rho+ P/c^2)}{1-\lvert u\rvert^2/c^2}\hspace{.5mm} u \otimes u\big) +\nabla P=0.
	\end{cases}
	\end{equation}Here and throughout,  $\rho \geq 0$ denotes the mass-energy density, $u=(u^{(1)},..., u^{(d)} )^\top \in \mathbb{R}^d$ denotes the relativistic velocity, $d\geq 1$ the dimension of the space, $c>0$ a large constant corresponding to  the speed of light, $P$ the pressure of the fluid, $x=(x^{(1)},..., x^{(d)} )^\top \in \mathbb{R}^d$ the Eulerian spatial coordinate and finally $t \in \mathbb{R}_{\geq 0}$ denotes the time coordinate.	The  constitutive relation $P = P(\rho)$  considered throughout most of this paper  is 
	\be \label{P} P(\rho) = k^2 \rho^\gamma, \ee   for  fixed  constants  $k>0$ and  $\gamma \geq 1$. 
	
	\vspace{3mm}
	Abiding by the fact that the theory of special relativity, in a regime of low velocities, should reduce to the classical Newtonian theory, the system $\eqref{RE}$ formally reduces to the classical $d$-dimensional isentropic compressible Euler equations (CE) when  $c$ approaches infinity:
	\be  \label{CE}\begin{cases} 
		\rho_t +\text{div}(\rho u ) = 0, \\
		
		(\rho u)_t +\text{div}(\rho u \otimes u) + \nabla P= 0.
	\end{cases}\ee

	\vspace{3mm}
	
	It is well-known that, for nonlinear hyperbolic conservation laws, a singularity  can form in finite time from initial compression  no matter how
	small or smooth the data are. Classical results including 
	Liu \cite{Liu1}, Li-Zhou-Kong \cite{LiZhouKong2} et al confirm that when the  initial data are
	small smooth perturbations of
	constant states, a blowup in the gradient of the solutions occurs in finite time \textit{if and only if} the initial data contain 
	any compression in a genuinely nonlinear characteristic field.
	\vspace{3mm}

	A natural follow-up question is whether  this dichotomy persists, at least for archetypal systems of conservation laws such as the (1+1)-dimensional relativistic Euler equations \eqref{RE} (resp.the classical compressible Euler equations \eqref{CE}), when one passes to the framework of large data problems. It turns out, as we shall promptly explain in this work, that one of the key issues towards establishing a similar dichotomy for large data is finding an
	effective way to obtain sharp enough control on the lower bound of the mass-energy  density (resp. the mass density). 
	
	For the $1$-dimensional classical  compressible Euler equations \eqref{CE},  some important  progress has been achieved for large data problems.  When $\gamma\ge 3$  in the  pressure law (\ref{P}),   the argument in Lax \cite{Lax} for general $2\times 2$ symmetric hyperbolic systems  can be applied as is to the  large data problem for the isentropic  Euler equations (see Section \ref{section4} of the current paper for CE)\footnote{See  also a generalization to full Euler equations by Chen-Young-Zhang \cite{ChenYoungZhang}.}.  What renders this possible is that, when $\gamma \geq 3$, a lower bound control for the density is not needed.  Therefore, the essence of the problem is to establish the finite time singularity formation for the compressible Euler equations in the most physically relevant case $1<\gamma<3$. 
	For piecewise Lipschitz continuous solutions, in an interesting paper, Lin  \cite{Lin2} argues  that the density has a (sharp) $O(1+t)^{-1}$ lower bound and proceeds to infer the corresponding global well-posedness. However satisfying, Lin's result comes with an important caveat: It only applies to initial data that are purely rarefactive, i.e. devoid of any compression. For general $C^1$ solutions including compressions in the solution, further novelties were required. In a recent paper  \cite{ChenPanZhu}, Chen-Pan-Zhu find an $O(1+t)^{-\frac{4}{(3-\gamma)}}$ lower bound when $1<\gamma<3$ \textit{and when the data is uniformly away from vacuum}. This result helps them to prove that gradient blowup of $\rho$ and/or ${u}$ happens in finite time \textit{if and only if} the initial data are forward or backward compressive somewhere, thus establishing the same dichotomy observed for small initial data \cite{LiZhouKong2,Liu1}.  Some further developments were achieved in \cite{ChenPanZhu2} where,  for general Lipschitz continuous solutions of (\ref{CE}) with $1<\gamma<3$, the authors improve the lower bound on the density from $O((1+t)^{-\frac{4}{(3-\gamma)}})$ in \cite{ChenPanZhu}  to the optimal order of $O(1+t)^{-1}$.  Finally, in Chen-Chen-Zhu \cite{ChenChenZhu} the authors  provide a new method to extend the theory to more general initial density profiles including possible far field vacuum, such as when $\rho(0,x)\in C^1(\mathbb{R})\cap L^1(\mathbb{R})$. Also, the authors provide the first global continuous non-isentropic solution including weak compression, using a new method involving solving an inverse Goursat problem. The solution they obtain  is almost classical,  except on two  characteristics, along which the solution has a  weak discontinuity (continuous but non-smooth).  		There is a large literature of works on sufficient conditions for formation of singularities in solutions to the compressible Euler equations and systems of hyperbolic conservation laws in multiple space dimensions. See a brief list in \cite{Christodoulou, ChristodoulouMiao, Dafermos, LiZhu,  Luk,   MakinoUkaiKawashima,    Rammaha, Sideris}.

	\vspace{3mm}
	
	In contrast, for the (1+1)-dimensional relativistic Euler equations, studies on the singularity formation and global well-posedness of  solutions with large data are very few.  Under the assumption that $P=\sigma^2 \rho$ for some positive constant $\sigma<c$, Smoller-Temple establish in their seminal work \cite{TempleSmoller} the existence of a global BV weak solution to the Cauchy problem, by crucially noticing that the shock curves for \eqref{RE}  in one dimension satisfy very strong geometric properties. For a general $\gamma$-law $P(\rho)$, Chen \cite{Chen} studies the corresponding Riemann problem. Later, Hsu-Lin-Markino  \cite{Hsu} establish the existence of global $L^\infty$ weak solutions with initial data containing the vacuum state. 
	For smooth solutions, Ruan-Zhu \cite{RuanZhu} first prove the global well-posedness of $C^1$ solutions with large data to the Cauchy problem for the (1+1)-dimensional relativistic Euler equations if the initial data do not have compression (see Definition \ref{def1}). For the (d+1)-dimensional ($\text{d} \geq 1$) relativistic fluids, Pan-Smoller \cite{PanSmoller} introduce  two sufficient conditions for the formation of singularities of smooth solutions: the initial data are compactly supported, or the radial component of the initial generalized momentum is sufficiently large.  However, for  the sufficient and necessary condition for singularity formation in $C^1$ solutions with large data, according to our discussion in Section \ref{section5},  due to lack of the lower bound estimate for the mass-energy density,  the problem has remained hitherto unexplored. The current paper addresses this  problem for the system \eqref{RE} in the $(1+1)$-dimensional spacetime setting, which  is a solid step in the study of relativistic Euler equations. Moreover,  for the $(3+1)$-dimensional relativistic fluids,  under the assumption that the initial mass-energy density vanishes in some open domain, we give two sufficient conditions for the regular solution to blow up in finite time, no matter how small the initial data are. Compared with the ones in \cite{PanSmoller},    we have removed the crucial assumptions that the initial data are compactly supported via an interesting observation for the multi-dimensional relativistic fluids: the invariance of the mass-energy density's  centroid.       We also do some interesting studies on the asymptotic behavior of the  relativistic velocity, which  tells us that  one can not obtain a global regular solution whose $L^\infty$ norm of $u$ decays to zero as time $t$ goes to infinity. Some interesting works on the study of the smooth solutions with vacuum for multi-dimensional relativistic fluids can be found in \cite{GengLi, Hadzic, Jang, Todd0, Todd1, Todd2,  ToddZhu1, ToddZhu2}.

	\vspace{3mm}

	Our paper is divided into 8 sections. In Section \ref{section2}, we introduce some basic notations and equations. In Section \ref{section3}, we state the main results.
	In Section \ref{section4}, for the smooth solutions with large data of  the 1-dimensional (classical) isentropic Euler equations  in Eulerian coordinates,  we present our novel idea for obtaining a time-dependent mass density   lower bound. The idea we present involves a careful study of the difference of the two Riemann invariants (and the study of certain weighted gradients of them). Naturally, why a new idea is needed in the first place is something that requires explanation. The reason is that in this problem, unlike the classical Euler equations, the introduction of Lagrangian coordinates cannot demystify the mathematical structure of the  system under study in the same, efficient way it did in \cite{ChenPanZhu, ChenPanZhu2, ChenChenZhu}. In other words, upon a thorough read of \cite{ChenPanZhu, ChenPanZhu2, ChenChenZhu}, one can see that the relation
	$$
	(1/\rho)_t=\frac{r_x+s_x}{2},
	$$
	where $r$ and  $s$ are the Riemann invariants of the so-called $p$-system in the Lagrange coordinate setting, is the cornerstone of the argument for obtaining the desired lower bound estimate. In the Eulerian setting this relation, or others of similar simplicity, are unavailable. Indeed, here one has
	$$
	(1/\rho)_t=-u (1/\rho)_x+(1/\rho)u_x,$$
	and therefore in order to get the mass density lower bound, i.e., $(1/\rho)'s$ upper bound, we should first have the upper bound of $-u (1/\rho)_x+(1/\rho)u_x$, which seems hard to obtain. 
	In Section \ref{section5}, combining an elaborate argument on a particular ODE inequality and introducing the crucial artificial quantity
	\be \label{Yquantity} \left( \frac{k \rho^{(\gamma-1)/2}/c}{\sqrt{1 + \frac{k^2  \rho^{\gamma-1}}{c^2}}}\right)^{\frac{3-\gamma}{2\gamma-2}}\big(1+ \frac{k^2 \rho^{\gamma-1}}{c^2}\big)^{\frac{\gamma+1}{4\gamma-4}} := \mathcal{Y},\ee 
	we apply our idea   to get a lower bound estimate for the  mass-energy density of  the (1+1)-dimensional relativistic Euler equations.  Ultimately, for $C^1$ solutions with large data and uniformly positive initial mass-energy density   of the (1+1)-dimensional relativistic Euler equations, we give a necessary and sufficient condition  for  the formation of singularities    in finite time.     	In Section \ref{section6},  we shift attention to (3+1)-dimensional relativistic fluids. Via introducing the particle number $n(\rho)$, we give a clear description on the time evolution of the vacuum boundary. Then we give  the proof for the corresponding two sufficient conditions of singularity formation.
	In Section \ref{section7}, we provide an extension of the 1-dimensional singularity formation results to more general pressure laws.  Finally, we include an appendix to show the related local-in-time well-posedness in multi-dimensional spacetime that is used in our paper.

	\section{Basic setup} \label{section2} \quad Before introducing the main results of this paper, we provide some equations and estimates for $C^1$ solutions of \eqref{RE}-\eqref{P} or \eqref{CE} with \eqref{P}, together with initial data
	\be
	(\rho,u)|_{t=0}=(\rho_0, u_0)(x) \label{initial1} \quad \text{for} \quad x\in \mathbb{R}^d,
	\ee
	for future reference, where $d=1$ or $3$.

	\subsection{(1+1)-dimensional Relativistic Euler equations}\label{subsection2.1}

	\smallskip

	\quad Let $d=1$ in \eqref{RE} and \eqref{initial1}.	
	We first define the $C^1$ solutions as follows:
	\begin{definition}\label{ 2.1-1} Let $T>0$ be some  time. The pair  $( \rho(t,x),u(t,x))$ is called a  $C^1$  solution to the relativistic Euler equations \eqref{RE}-\eqref{P} on $(0, T ) \times \mathbb{R}$  if
		\begin{equation*}\begin{split}
		&\rho>0, \quad \rho \in C^1([0, T )\times \mathbb{R}),\  u  \in C^1([0, T )\times \mathbb{R}),\end{split}
		\end{equation*}
		and the equations \eqref{RE}-\eqref{P} are satisfied in the pointwise sense on $(0,T) \times \mathbb{R}$. It is called a $C^1$ solution to the Cauchy problem \eqref{RE}-\eqref{P} with \eqref{initial1} if it is a $C^1$ solution to the equations \eqref{RE}-\eqref{P} on $(0, T ) \times \mathbb{R} $ and admits the initial data  \eqref{initial1} continuously. \end{definition}
	
	\par \noindent It is well-known\footnote{See \cite{RuanZhu} and the references cited therein.} that  there exists a  local-in-time $C^1$ solution $(\rho, u)$ in $[0,T] \times \mathbb{R}$   for some $T>0$, when 
	\begin{equation}\label{c1data}
	\inf_{x\in \mathbb{R}}\rho_0>0,\quad 	(\rho_0, \ u_0) \in C^1(\mathbb{R}).
	\end{equation}

	We proceed with a rudimentary analysis of the Cauchy problem \eqref{RE}-\eqref{P} with  \eqref{initial1}. First, the two  eigenvalues $\lambda_1$ and $\lambda_2$ of equations   \eqref{RE}-\eqref{P}   can be given by
	\be \label{lambda12} \lambda_1 = \frac{u- \sqrt{P'}}{1- \frac{u\sqrt{P'}}{c^2}}\qquad \text{and} \qquad  \lambda_2 = \frac{u+ \sqrt{P'}}{1+ \frac{u\sqrt{P'}}{c^2}}. \ee 
	We denote the directional derivatives as
	\be \label{lambda13} \prime = \partial_t + \lambda_1 \partial_x, \quad  \backprime = \partial_t +\lambda_2 \partial_x\ee
	along two characteristic directions
	\be \label{lambda14} 
	\frac{dx^1}{dt}=\lambda_1\quad \text{and} \quad \frac{dx^2}{dt}=\lambda_2,
	\ee
	respectively and introduce the corresponding Riemann variables
	\begin{gather}
	w = \frac{c}{2}\hspace{.5mm} \text{ln}\Big(\frac{c+u}{c-u}\Big) + \int_0^\rho \frac{\sqrt{P'(\sigma)}}{\sigma + \frac{P(\sigma)}{c^2}} \text{d}\sigma, \label{ru}\\[8pt]
	\displaystyle
	z=  \frac{c}{2}\hspace{.5mm} \text{ln}\Big(\frac{c+u}{c-u}\Big) - \int_0^\rho \frac{\sqrt{P'(\sigma)}}{\sigma + \frac{P(\sigma)}{c^2}} \text{d}\sigma.\label{srho}
	\end{gather} 
	Then,  it is easy to see that $w$ and $z$ satisfy
	\be  \label{REODE}    w^\backprime = 0 \quad \text{and} \quad z^\prime = 0. \ee
	
	\par \noindent Let $h_1$ and $h_2$ be functions satisfying 
	
	\be    \label{h1h2def}   h_{1w} = \frac{\lambda_{1w}}{\lambda_1 - \lambda_2}, \hspace{2mm} h_{2z} = \frac{\lambda_{2z}}{\lambda_2 - \lambda_1}. \ee  Define $\alpha = z_x, \beta = w_x$ and introduce
	
	\be \label{zwode} \xi= \e^{h_1}\alpha, \hspace{2mm} \zeta= \e^{h_2} \beta. \ee

	\par \noindent 	We continue with a simplification for $s(\rho)$,  as can be found for example in \cite{Chen},
	
	\be \label{srho11} s(\rho)=  \frac{2c \sqrt{\g}}{\g-1} \text{\text{Arctan}}\big( \frac{k \rho^{(\g-1)/2}}{c}\big).        \ee 
	
	\par \noindent	Another important calculation is the following expression for $\sqrt{P'(\rho)}$ in terms of the Riemann invariants:
	
	\be \label{Pprimeformularel} \sqrt{P'(\rho)} = k \sqrt{\gamma} \rho^{(\gamma-1)/2 } = c\sqrt{\gamma}\hspace{1mm} \text{\text{Tan}}\Big( \frac{(w-z)(\g-1)}{4c \sqrt{\g}} \Big). \ee

	Finally, we define the compression and rarefaction characters.
	\begin{definition}
		\label{def1}
		The local {$R/C$} character for a classical solution of \eqref{RE}-\eqref{P} with  \eqref{initial1} is
		\[\begin{array}{lllll} 
		\text{Forward}& \ \  $R$\ \  \text{iff} \ \ w_x>0;\quad \text{Forward} \ \ \ $C$ \ \ \text{iff} \ \ w_x<0;\\
		\text{Backward}& \ \ $R$ \ \ \text{iff}\ \ z_x>0;\quad \text{Backward} \ \  $C$ \ \ \text{iff} \ \ z_x<0.
		\end{array}\]
	\end{definition}
	Although this definition was not clearly provided in Lax \cite{Lax},
	his result on some cases of $2\times2$ hyperbolic conservation laws
	can be explained as follows: a singularity forms in finite time if and only if
	there exists some backward or forward compression under Definition \ref{def1}.
	According to the results obtained in this paper, we  see that  this definition of compression and rarefaction
	gives a clean cut on the singularity formation.
	
	\subsection{1-dimensional classical compressible  Euler equations}\label{subsection2.2}
	
	\quad Let $d=1$ in \eqref{CE} and \eqref{initial1}.
	To make the corresponding statement precise, we first define the $C^1$ solutions as follows:
	\begin{definition}\label{ 2.1-2} Let $T>0$ be some  time. The pair $( \rho(t,x),u(t,x))$ is called a $C^1$  solution to the non-relativistic Euler equations \eqref{CE} with \eqref{P} on $(0, T ) \times \mathbb{R}$  if
		\begin{equation*}\begin{split}
		&\rho>0, \quad \rho \in C^1([0, T )\times \mathbb{R}),\  u  \in C^1([0, T )\times \mathbb{R}),\end{split}
		\end{equation*}
		and the equations \eqref{CE} with \eqref{P} are satisfied pointwise on $(0,T) \times \mathbb{R}$. It is called a $C^1$ solution to the Cauchy problem \eqref{CE} with \eqref{P}and \eqref{initial1} if it is a $C^1$ solution to the equations \eqref{CE} with \eqref{P} on $(0, T ) \times \mathbb{R} $ and admits the initial data  \eqref{initial1}      continuously. \end{definition}
	
	As in the relativistic case, it is well-known that  there exists a  local-in-time $C^1$ solution $(\rho, u)$ in $[0,T] \times \mathbb{R}$   for some $T>0$, when  (\ref{c1data}) is satisfied.
	We proceed with a rudimentary analysis of the Cauchy problem \eqref{CE} with \eqref{P} and \eqref{initial1}. First, the two  eigenvalues $\tilde{\lambda}_1 $ and $\tilde{\lambda}_2$ of equations   \eqref{CE} with \eqref{P}   can be given by
	
	\be \label{lambdat12} \tilde{\lambda}_1 = u- \sqrt{P'}, \hspace{2mm} \tilde{\lambda}_2 = u + \sqrt{P'}. \ee
	We denote the directional derivatives as
	\be \label{lambdat13} \partial_-= \partial_t + \tilde{\lambda}_1\partial_x, \quad  \partial_+= \partial_t +\tilde{\lambda}_2 \partial_x \ee
	along two characteristic directions
	\be \label{lambdat14} 
	\frac{dy^1}{dt}=\tilde{\lambda}_1\quad \text{and} \quad \frac{dy^2}{dt}=\tilde{\lambda}_2,
	 \ee
	respectively, and introduce the corresponding Riemann variables
	\begin{gather}
	\displaystyle
	 \tilde{w}= u + \int_0^\rho \frac{\sqrt{P'(\sigma)}}{\sigma}\hspace{.5mm} \text{d}s, \hspace{2mm} \tilde{z}= u - \int_0^\rho \frac{\sqrt{P'(\sigma)}}{\sigma}\hspace{.5mm} \text{d}\sigma.
	\end{gather} 
	Then, it is easy to see that $\tilde{w}$ and $\tilde{z}$ satisfy
	\be  \label{CEODE}   \partial_+\tilde{w} = 0 \quad \text{and} \quad \partial_-\tilde{z} = 0. \ee 
	
	\par \noindent Let $\tilde{h}_1, \tilde{h}_2$ be functions satisfying
	
	\be \label{hfunctionclas} \tilde{h}_{1\tilde{w}} = \frac{\tilde{\lambda}_{1\tilde{w}}}{\tilde{\lambda}_1- \tilde{\lambda}_2}, \hspace{2mm} \tilde{h}_{2\tilde{z}} = \frac{\tilde{\lambda}_{2 \tilde{z}}}{\tilde{\lambda}_2- \tilde{\lambda}_1}.\ee 
	Define $\tilde{\alpha} = \tilde{z}_x, \hspace{.5mm} \tilde{\beta}= \tilde{w}_x$ and introduce
	
	\be \label{phipsi}\phi = \e^{\tilde{h}_1} \tilde{\alpha} , \hspace{2mm} \psi= \e^{\tilde{h}_2}  \tilde{\beta}.         \ee

	Finally, we define the compression and rarefaction characters.
	\begin{definition}
		\label{def1CE}
		The local {$R/C$} character for a classical solution of \eqref{CE} with \eqref{P} and \eqref{initial1} is
		\[\begin{array}{lllll} 
		\text{Forward}& \ \ $R$ \ \  \text{iff}  \ \  \tilde{w}_x>0; \ \  \text{Forward} \ \ \ $C$ \ \  \text{iff} \ \ \tilde{w}_x<0;\\
		\text{Backward}&  \ \  $R$  \ \  \text{iff} \ \  \tilde{z}_x>0;\ \ \text{Backward} \ \ $C$ \ \ \text{iff} \ \ \tilde{z}_x<0.
		\end{array}\]
	\end{definition}
	
	\par \noindent According to the results obtained \cite{ChenPanZhu, ChenPanZhu2, ChenChenZhu}, we know that  this definition on the compression and rarefaction
	gives a clean cut on the singularity formation.	
	\subsection{(3+1)-dimensional Relativistic Euler equations with vacuum}

	\quad Let $d=3$ in \eqref{RE} and \eqref{initial1}.	To make the corresponding statement precise, we give the definition of regular  solutions considered in this paper:
	\begin{definition}\label{d1}
		Let $T> 0$ be some time. The pair $(\rho(t,x),u(t,x))$ is called a regular solution in $ [0,T]\times \mathbb{R}^3$  to the  Cauchy problem  (\ref{RE})-(\ref{P}) with \eqref{initial1}  if $(\rho,u)$ satisfies this problem in the sense of distributions and:
		\begin{equation*}\begin{split}
		&(\textrm{A})\quad  \rho\geq 0, \  (\rho,u)\in C^1([0,T]\times \mathbb{R}^3);\\[2pt]
		&(\textrm{B})\quad |u|<c,\quad  \sqrt{P'(\rho)}<c;\\[2pt]
		&(\textrm{C})\quad u_t+u\cdot\nabla u =0\quad  \text{whenever} \quad  \rho(t,x)=0.
		\end{split}
		\end{equation*}
	\end{definition}

	\vspace{3mm}
	
	Let $x(t;x_0)$ be the particle path starting from $x_0$ at $t=0$, i.e.,
	\begin{equation}\label{gobn}
	\frac{d}{\text{d}t}x(t;x_0)=u(t,x(t;x_0)),\quad x(0;x_0)=x_0.
	\end{equation}In the rest of this section, we will use the following useful physical quantities in some domain $\Omega \subset \mathbb{R}^3$:
	\begin{align*}
	&m(t)=\int_{\Omega}\hat{\rho}(t,x)\text{d}x \quad \textrm{(total energy)},\\[2pt]
	&X^*(t)=\frac{\int_{\Omega} \hat{\rho} x \hspace{.5mm}\text{d}x}{m(t)}\ (\textrm{centroid} \ \textrm{of} \  \Omega),\\[2pt]	
	&M(t)=\int_{\Omega}\hat{\rho}(t,x)|x|^2\text{d}x \quad \textrm{(second moment)},\\[2pt]	
	&F(t)=\int_{\Omega} \tilde{\rho}(t,x)u(t,x)\cdot x \hspace{.5mm} \text{d}x \quad \textrm{ (radial component of momentum)},\\[2pt]
	&\mathbb{P}(t)=\int_{\Omega} \tilde{\rho}(t,x)u(t,x) \hspace{.5mm} \text{d}x \quad \textrm{(momentum)},
	\end{align*}
	where $\hat{\rho}$ and $\tilde{\rho}$ are given by 
	$$
	\hat{\rho}= \frac{\rho + P \lvert u\rvert^2/c^4}{1-\lvert u\rvert ^2/c^2}=\frac{1}{c^2}\tilde{\rho}\lvert u\rvert^2+\rho  \quad \text{and} \quad \tilde{\rho}=\frac{(\rho+ P/c^2)}{1-\lvert u\rvert ^2/c^2}=\hat{\rho}+\frac{P}{c^2}.$$
	
	Based on the  physical quantities introduced  in this subsection,  we define  one solution classe as follows:
	\begin{definition}\label{d2}
		Let $T>0$ be some time and $\Omega=\mathbb{R}^3$ in the definitions of $m(t)$ and $\mathbb{P}(t)$. The pair $(\rho(t,x),u(t,x))$    is said to be in  the solution class $ D(T)$  of   the Cauchy problem  \eqref{RE}-\eqref{P} with  \eqref{initial1}        if  $(\rho,u)$  satisfies this problem in the sense of distributions and 
		\begin{equation*}\begin{split}
		&(\textrm{A})\quad	\rho \geq 0, \ \ \  (\rho,u) \in C^1([0, T )\times \mathbb{R}^3);\\[2pt]	   
		&(\textrm{B})\quad |u|<c,\quad  \sqrt{P'(\rho)}<c; \\[2pt]	         	
		&(\textrm{C})\quad \text{Conservation of total energy:}\quad 0<m(0)=m(t)<\infty \ \text{for  any}  \ t\in [0,T];\\[2pt]
		&(\textrm{D})\quad \text{Conservation of momentum:}\quad   0<|\mathbb{P}(0)|=|\mathbb{P}(t)|<+\infty \ \text{for  any} \ \ t\in [0,T].
		\end{split}
		\end{equation*}
	\end{definition}


The corresponding local-in-time well-posedness of smooth solutions defined  in Definitions \ref{d1} and \ref{d2} has   been established  by  Lefloch-Ukai \cite{LeflochUkai} (see also our Appendix).

	\section{Statement of  main results}\label{section3}
	
	\qquad In this section, we will state our main results in the following two subsections.
	\subsection{(1+1)-dimensional case}
	
	\qquad   Let $d=1$ in \eqref{RE} and \eqref{initial1}.  From now on, we make the following assumption throughout the rest of this paper:
	\begin{assumption}\label{p-assumption} Assume that\footnote{For the  polytropic gas, the condition $\sqrt{P^\prime\left(\Jrel^{-1}\left(\frac{w_{\max}-z_{\min}}{2}\right)\right)} < c$ can be read as 	$$
			w_{\text{max}} - z_{\text{min}} < \frac{4 c \sqrt{\gamma}}{\gamma-1} \text{\text{Arctan}}\big( \frac{1}{\sqrt{\g}} \big).$$ } 
		\begin{equation*}\begin{split}
		&\inf_{x\in \mathbb{R}}(w_0 - z_0)>0,\quad \sqrt{P^\prime\left(\Jrel^{-1}\left(\frac{w_{\max}-z_{\min}}{2}\right)\right)} < c,\\		
		& (w_0, \ z_0) \in C^1(\mathbb{R}),\quad \|(w_0, z_0)\|_{C^1(\mathbb{R})}\le M_0, 
		\end{split}
		\end{equation*}	
		for some  constant $M_0>0$, where 
		\begin{equation*}\begin{split} &(w_0,  z_0)(x)=(w(0,x),  z(0,x)),\quad   \Jrel(x) = \int_0^x \frac{\sqrt{P'(\sigma)}}{\sigma+ \frac{P(\sigma)}{c^2}}\text{d}s,\\
		&w_{\text{max}}= \sup \{w_0(x) \hspace{1mm} \mid x \in \mathbb{R} \}, \hspace{1mm} z_{\text{min}}= \inf \{ z_0(x) \hspace{1mm} \mid x \in \mathbb{R} \},
		\end{split}\end{equation*}
 and $\Jrel^{-1}$ denotes the  inverse function of $\Jrel(x)$.

	\end{assumption}

	\begin{theorem}\label{1RE}
		For polytropic gas $\gamma>1$ in \eqref{P}, if $(w_0(x), z_0(x))$ satisfy  the conditions in Assumption \ref{p-assumption}, then the Cauchy problem \eqref{RE}-\eqref{P} with  \eqref{initial1} has a unique 
		global-in-time $C^1$ solution if and only if 
		\be\label{p_lemma_con}
		w_x(x,0)\geq0 \ \ \text{and}\ \   z_x(x,0)\geq 0,\ \text{for all} \ \ x\in\mathbb{R}\,.
		\ee
	\end{theorem}
	
	For more general pressure laws $P(\rho)$, 	first we give the following two assumptions:
	
	\begin{assumption}\label{assumption2}
	\be \begin{split} &P(\rho) >0, \quad  P'(\rho) > 0, \quad  P''(\rho) > 0,\\
	& P(0)=0, \quad \lim_{\rho\rightarrow \infty}\int_0^\rho \frac{\sqrt{P'(\sigma)}}{\sigma + \frac{P(\sigma)}{c^2}}\hspace{.5mm} d\sigma<\infty.
		\end{split}
	\ee
	\end{assumption}
	
	\begin{assumption}\label{assumption3}
		There exists a positive constant $A$ such that, for all $\rho>0$, there holds 
		
		\be \rho^8 \big( (5+A)P''(\rho)^2 - 4 P'(\rho) P'''(\rho)\big) + (4A-4) \big( \rho^6 P'(\rho)^2 + \rho^7 P'(\rho) P''(\rho) \big) \geq 0. \ee
	\end{assumption}	
	
	Now we can state the following theorem:
	
	\begin{theorem}\label{2RE}
		For a general pressure law $P(\rho)$ satisfying Assumptions \ref{assumption2}-\ref{assumption3}, if $(w_0(x), z_0(x))$ satisfy  the conditions in Assumption \ref{p-assumption}, then the Cauchy problem \eqref{RE}-\eqref{P} with  \eqref{initial1} has a unique 
		global-in-time $C^1$ solution if and only if 
		\be\label{p_lemma_con}
		w_x(x,0)\geq0 \ \ \text{and}\ \   z_x(x,0)\geq 0,\ \text{for all} \ \ x\in\mathbb{R}\,.
		\ee
		Here, the corresponding definitions of $(w,z)$ under a general pressure law can be found in Section 7.2.
	\end{theorem}

	The proof of Theorems \ref{1RE}-\ref{2RE} can be found in Sections  \ref{section5} and \ref{section7} respectively.  We make some necessary remarks on our conclusions at this point:
	\begin{remark}
		The  assumption 
	$\sqrt{P^\prime\left(\Jrel^{-1}\left(\frac{w_{\max}-z_{\min}}{2}\right)\right)} < c$
		is imposed for one to be able to show  that the local sound speed $\sqrt{P'(\rho)}$ is bounded away from the light speed $c$ for  $C^1$ solutions, which is natural in the sense of physics.
	\end{remark}
	\begin{remark}
		Our conclusion establishes a necessary and sufficient condition for the formation of singularities for the $(1+1)-$dimensional relativistic Euler equations, namely the existence of compression in the initial data, which gives a complete picture on the formation of singularities for the Cauchy problem of the system $\eqref{RE}$ in $(1+1)$-dimensional Minkowski spacetime.
	\end{remark}
	\begin{remark}
		A question of physical significance is to determine the type of singularity that forms. It is generally expected that the discontinuities developed, among others, in the works of Chen-Pan-Zhu \cite{ChenPanZhu}, Lax \cite{Lax} and the current work are indeed discontinuity singularities, i.e. shock waves. The proof of disproof of such a fact, in full generality, remains an open problem. The best partial results known to the authors, however can be found in Kong \cite{Kong}. There, for a general class of strictly hyperbolic $2\times 2$ systems with two genuinely nonlinear characteristic fields, it is shown that if a singularity forms then it develops as a shock wave if either
		
		\begin{itemize}
			\item one of the two Riemann invariants, w or z, is initially a constant\footnote{Thus the study of the system reduces to the study of a scalar conservation law, which in general has a complete theory. See also Lebaud \cite{Lebaud} for the 1-D classical Euler equations.}, or
			\item certain a priori conditions, essentially quantitative bounds on the size of the derivatives, hold at the blow-up point; conditions which are, however, difficult to verify.
		\end{itemize}Thus, our theorem above along with Kong's result readily implies that, if one of the two initial data variables $w_0, z_0$ is a constant, a shock forms if and only if there exists initial compression in the non-constant Riemann invariant variable.  In any case, further study on the type of singularities obtained in this work and several others promises to be a meaningful and interesting direction for research.
	\end{remark}
	
	
	
	
	\subsection{(3+1)-dimensional case}

	\qquad   Let $d=3$ in \eqref{RE} and \eqref{initial1}.  We present two scenarios for singularity development in finite time from initial data with vacuum in some local domain.  The first one is the  isolated mass group:
	
	\begin{definition}[\textbf{Isolated mass group}]\label{local}The initial data pair $(\rho_0(x),u_0(x))$ is said to have an isolated mass group $(A_0,B_0)$ if there exist two smooth,  bounded and connected open sets  $A_0 \subset \mathbb{R}^3$ and $B_0\subset \mathbb{R}^3$ satisfying
		\begin{equation} \label{eq:12131}
		\begin{cases}
		\displaystyle
		\overline{A}_0 \subset B_0 \subseteq  B_{R_0} \subset \mathbb{R}^3;         \\[10pt]
		\displaystyle
		\rho_0(x)=0,\ \forall \  x\in B_0 \setminus A_0, \quad \int_{A_0} \rho_0(x) \text{d}x>0; \\[10pt]
		\displaystyle
		u_0(x)|_{\partial A_0}=\overline{u}_0,\quad \int_{A_0} \tilde{\rho}(0,x)u_0 \hspace{.5mm} \text{d}x=0,
		\end{cases}
		\end{equation}
		for some positive constant $R_0$ and  constant vector $\overline{u}_0 \in \mathbb{R}^3$, where $B_{R_0}$ is  the ball centered at the origin with radius $R_0$.
	\end{definition}

	\vspace{3mm}
	
	Our first blowup result shows that the existence of an isolated mass group in the initial data guarantees the finite time singularity formation of regular solutions.
	\begin{theorem}[\textbf{Blow-up by isolated mass group}]\label{coo2}\
		If the initial data $(\rho_0,u_0)(x)$ have an isolated mass group $(A_0,B_0)$,
		then the  regular solution $(\rho, u)(t,x)$ in $[0, T_m]\times\mathbb{R}^3$  defined  in Definition \ref{d1} with maximal existence time $T_m$  to  the Cauchy problem  \eqref{RE}-\eqref{P} and \eqref{initial1}    blows up in finite time, i.e.,
		$\
		T_m<+\infty.
		$
	\end{theorem}

	For the second scenario, we explore the hyperbolic structure for the system in a vacuum region. For this purpose, we  introduce the following concept:
	
	\begin{definition}[\textbf{Hyperbolic singularity set}]\label{bugers}
		We define the smooth,  open set $V \subset \Omega$ as a hyperbolic singularity set, if $V$ and $(\rho_0,u_0)$ satisfy
		\begin{equation} \label{eq:12131sss}
		\begin{cases}
		\displaystyle
		\rho_0(x)=0, \ \forall \ x\in V;\\[10pt]
		\displaystyle
		Sp(\nabla u_0) \cap \mathbb{R}^-\neq\  \emptyset,\quad  \forall  \ x \in V,
		\end{cases}
		\end{equation}
		where we denote by $Sp(\nabla u_0(x))$ the spectrum of the Jacobian matrix of $u_0$.
	\end{definition}
	
	Then we show the following:
	\begin{theorem}[\textbf{Blow-up by the hyperbolic singularity set}]\label{th333}If the initial data $(\rho_0,u_0)(x)$ have a hyperbolic singularity set $V$,
		then the  regular solution $(\rho, u)(t,x)$ in $[0, T_m]\times\mathbb{R}^3$ defined  in Definition \ref{d1}  with maximal existence time $T_m$ to  the Cauchy problem  \eqref{RE}-\eqref{P} and \eqref{initial1}   blows up in finite time, i.e.,
		$\
		T_m<+\infty.
		$	\end{theorem}

	A  natural question to ask is   whether the local solution in \cite{LeflochUkai} can be extended globally in time if we can identify some initial data which avoid the above two blow-up mechanisms and what  the large time behavior is. In the following theorem, we give a very interesting observation for the solution's asymptotic behavior.

	\begin{theorem}\label{th:2.20}
		Let  $\gamma>1$ in \eqref{P}. For   the Cauchy problem  \eqref{RE}-\eqref{P} with  \eqref{initial1}, there is no  solution $(\rho,u)\in D(\infty)$  satisfying 
		\begin{equation}\label{eq:2.15}
		\limsup_{t\rightarrow +\infty} \|u(t,x)\|_{L^\infty(\mathbb{R}^3)}=0.
		\end{equation}

		\end {theorem}
		

			

			
			\section{The mass density lower bound of  1-dimensional CE} \label{section4}
			\qquad  Let $d=1$ in \eqref{CE} and \eqref{initial1}.  Recall at this point the notation introduced in Subsection \ref{subsection2.2}. We dedicate this section to the presentation of the new approach for obtaining the crucial lower bound estimate on the mass density for the classical compressible Euler equations \eqref{CE}. The main idea is that, instead of obtaining a transport equation for $\rho$ and using it to obtain the estimate, we focus instead on the difference of the two Riemann invariants, in the classical case given by $\tilde{w}-\tilde{z}$. The function $\tilde{w}-\tilde{z}$ is an increasing function of $\rho$ and therefore control on $\tilde{w}-\tilde{z}$ translates to control on $\rho$. To best exhibit our approach, we apply it in the first subsection to the classical Euler equations. In the final subsection, we lay down the main argument for singularity formation in finite time, which is essentially that of \cite{Lax}. We explain then why a lower bound estimate is of such importance; and how we may use to conclude our argument.

			\subsection{The mass density lower bound estimate in the Eulerian setting }\label{subsection4.1} We begin by noticing that the weighted gradients satisfy certain Riccati equations.	
			\begin{lemma}\label{lemma41} For the $C^1$ solution of the system \eqref{CE}, the following Riccati ODEs hold:
				
				\be \partial_- \phi= -\big(\e^{-\tilde{h}_1}\tilde{\lambda}_{1\tilde z}\big)\phi^2, \hspace{2mm} \partial_+ \psi= -\big(\e^{-\tilde{h}_2}\tilde{\lambda}_{2\tilde w}\big)\psi^2. \ee 
				Moreover, let $y^i(t,y^i_0)$ $(i=1,2)$ be two characteristic curves (defined in Section 2.2) starting from $(0,y^i_0)$. One has 
				\[ \frac{1}{\phi(t,y^1(t,y^1_0))} = \frac{1}{\phi(0,y^1_0) } + \int_0^t \big(\mathrm{e}^{-\tilde{h}_1}\tilde{\lambda}_{1\tilde{z}}\big) (\sigma,y^1(\sigma,y^1_0))\hspace{.5mm} \text{d}\sigma.   \] 
				\[ \frac{1}{\psi(t,y^2(t,y^2_0))} = \frac{1}{\psi(0,y^2_0) } + \int_0^t \big(\mathrm{e}^{-\tilde{h}_2}\tilde{\lambda}_{2\tilde{w}}\big) (\sigma,y^2(\sigma,y^2_0))\hspace{.5mm} \text{d}\sigma.   \] 								
				
			\end{lemma}	
			
			\begin{proof} Differentiating the last equation of \eqref{CEODE} with respect to $x$, and  recalling the definition of $\tilde{h}_1$ from \eqref{hfunctionclas}, we arrive at
				\[  \partial_- \hspace{.5mm} \tilde{\alpha}+( \partial_- \tilde{h}_1 \hspace{.5mm}) \tilde{\alpha} + \tilde{\lambda}_{1\tilde{z}}\hspace{.5mm} \tilde{\alpha}^2=0 ,       \] 
				which, along with  $\phi = \mathrm{e}^{\tilde{h}_1}\tilde{\alpha}$, implies the desired ODE on $\phi$.  The formula of $\phi$ along the backward characteristic curve can be obtained by solving the 	Riccati ODE that we obtained. 			
				The proof for $\psi$ is similar, and here we omit the details.

			\end{proof}

			Our strategy will be to work towards obtaining a time-dependent lower bound on $\tilde w-\tilde z$. To achieve this, we must first rewrite $\tilde{\lambda}_1, \tilde{\lambda}_2$ in terms of $\tilde{w},\tilde{z}$:
			\be \begin{cases}\displaystyle \tilde{\lambda}_1 = \frac{\tilde{w}+\tilde{z}}{2}- \frac{(\tilde{w}-\tilde{z})(\g-1)}{4},\\[3pt]
				\displaystyle \tilde{\lambda}_2 = \frac{\tilde{w}+\tilde{z}}{2}+ \frac{(\tilde{w}-\tilde{z})(\g-1)}{4}. \end{cases} \ee 
			One can get that $\tilde{\lambda}_{1\tilde{z}} = \tilde{\lambda}_{2\tilde{w}} = \frac{\g+1}{4}$.
			Notice that in the classical case treated here, one can take \be\tilde{h}_1 = \tilde{h}_2 = \tilde{h} = \frac{\g-3}{2\g-2} \ln(\tilde w-\tilde z),\ee which implies that   \be \label{formofequations1} e^{-\tilde{h}_1}\tilde{\lambda}_{1\tilde z}= \e^{-\tilde{h}_2}\tilde{\lambda}_{2\tilde w}= \frac{\g+1}{4} (\tilde w-\tilde z)^{\frac{3-\g}{2\g-2}}. \ee 
			
			Based on these  observations and Lemma \ref{lemma41}, standard ODE theory then leads us to the following result:
			\begin{proposition}\label{propositionmaxclas}
				Denote
				\[ \widetilde{Q}_1 = max\begin{Bmatrix}0 ,\sup_x \phi(0,x) \end{Bmatrix},  \quad \widetilde{Q}_2 = max\begin{Bmatrix}0 ,\sup_x \psi(0,x) \end{Bmatrix}.       \]  For the $C^1$ solution of the system \eqref{CE},  there holds $  \phi \leq \widetilde{Q}_1, \hspace{1mm} \psi \leq \widetilde{Q}_2.  $
			\end{proposition}

			Finally, we can get the desired lower bound estimates of the mass density.
			\begin{lemma}\label{lowerboundofCE}
				let $y^i(t,y^i_0)$ $(i=1,2)$ be two characteristic curves starting from $(0,y^i_0)$.  	There holds 	
				\be \big(\e^{-\tilde h}\tilde{\lambda}_{1\tilde{z}}\big)(t, y^1(t,y^1_0)) \geq \frac{1}{C_1+ C_2 t}, \quad   \big(\e^{-\tilde h}\tilde{\lambda}_{2\tilde{w}}\big)(t, y^2(t,y^2_0)) \geq \frac{1}{C_1+ C_2 t},      \ee 	
				for positive constants $C_i$ $(i=1,2)$	 independent of the time.			
				
			\end{lemma}
			\begin{proof} According to \eqref{CEODE} , and the definition of $\psi$, one can obtain that 
				\be \label{lemmaequation1} \partial_-(\tilde w-\tilde z) = (\tilde w-\tilde z)_t + \tilde{\lambda}_1(\tilde w-\tilde z)_x = (\tilde {\lambda}_1 - \tilde {\lambda}_2) \tilde {w}_x=- \frac{\tilde{\lambda}_2 - \tilde{\lambda}_1}{\e^{\tilde{h}_2}} \psi.   \ee  
				which, along with the  Proposition \ref{propositionmaxclas} and the relation \eqref{formofequations1}, implies that  
				\be \label{lemmaequation2} \partial_-(\tilde w-\tilde z) \geq -   \frac{\tilde{\lambda}_2 - \tilde{\lambda}_1}{\e^{\tilde{h}_2}} \widetilde{Q}_2= -\frac{(\g-1) \widetilde{Q}_2}{2} \hspace{.5mm} (\tilde w-\tilde z)^{\frac{\g+1}{2\g-2}}. \ee 
				Denoting  $C'=\frac{(\g-1) \widetilde{Q}_2}{2}$, and integrating  \eqref{lemmaequation2} along $y^1(t,y^1_0)$ over $[0,t]$,  yields 
				\be (\tilde{w}-\tilde{z})(t, y^1(t,y^1_0)) \geq  \Big( \big((\tilde{w}-\tilde{z})(0,y^1_0)\big)^{\frac{2\gamma-2}{\gamma-3}}+ \frac{3-\gamma}{2\gamma-2}C't)             \Big)^{\frac{2\g-2}{\g-3}}.\ee
				In particular, taking \eqref{formofequations1} into account, one can  obtain the time-dependent lower  bound
				\be \big(\e^{-\tilde h}\tilde{\lambda}_{1\tilde{z}}\big)(t, y^1(t,y^1_0)) \geq \frac{1}{C_1 + C_2t}, \ee 
				for positive constants $C_i$ $(i=1,2)$	 independent of $t$.	Similarly, we can obtain the second  estimate. \end{proof}
			\begin{remark}
				Taking into account that $w-z = C \rho^{\frac{\gamma-1}{2}}$ for some  universal constant $C>0$ and Lemma \ref{lowerboundofCE}, we recover precisely the result of \cite{ChenPanZhu}: \[ \rho(t,x) \geq \left(\frac{1}{C_1 +C_2 t}\right)^{\frac{4}{3-\gamma}}.   \]
			\end{remark}
			\subsection{Formation of singularity}

			\qquad It is important at this stage to highlight the main mechanism that shall be used throughout the paper to obtain the formation of singularities. The argument within this subsection can be traced back to P.D. Lax in his 1964 paper on $2\times2$ systems. This is best described in the context of system \eqref{CE}. The Riccati type equations in Lemma \ref{lemma41} are precisely what gives us a clear passage to study the singularity
			formation and/or global existence of classical solutions for hyperbolic systems
			with two unknowns.

			Without loss of generality, we assume  that there exists a point $(0,y^1_0)$ on the initial data line $t=0$ such that $\phi(0,y^1_0)<0$.
			Then we see that a sufficient condition for the breakdown of the classical solution is
			\be   \label{singularitycond}  \int_0^\infty \big(\mathrm{e}^{-\tilde{h}_1}\tilde{\lambda}_{1\tilde{z}}\big) (t,y^1(t,y^1_0))\hspace{.5mm} \text{d}t = \infty,       \ee which, actually can be verified by the conclusions obtained in Lemma \ref{lowerboundofCE}. The proof of our main theorem in the next section essentially comes down to establishing a statement of the form \eqref{singularitycond} for the system of $(1+1)-$dimensional relativistic Euler equations.

			\section{Formation of singularities for the (1+1)-dimensional RE}\label{section5}
			\qquad In this section we shall lay down the proof of Theorem \ref{1RE}.     Let $d=1$ in \eqref{RE} and \eqref{initial1}.   As we mentioned in Subsection \ref{subsection2.1}, it is a well-known result that given initial data as in Theorem \ref{1RE}, there exists a $T \in (0,\infty)$ such that there exists a  local-in-time $C^1$ solution to the Cauchy problem $\eqref{RE}-\eqref{P}$ with \eqref{initial1}. We proceed by obtaining estimates on the solution in the slab $[0,T]\times \mathbb{R}$.

			\subsection{Preliminaries}\label{subsection5.1}
			
			\qquad Before giving the detailed proof, we first give several fundamental lemmata for the RE equations. 
			First, we show  that the relativistic fluid velocity $u$ is less than the light speed $c$.
			\begin{lemma}\label{lemma5.1}
				For the $C^1$ solution of the  Cauchy problem $\eqref{RE}$-$\eqref{P}$ with \eqref{initial1}, under the Assumption \ref{p-assumption}, 	the absolute value $\lvert u \rvert$ of the velocity function is uniformly bounded away from the light speed $c$.
			\end{lemma} \begin{proof} According to  \eqref{CEODE}, one can obtain that 
				\be \big\lvert \ln\Big(\frac{c+u}{c-u}\Big)\big\rvert = \big\lvert \frac{w+z}{c} \big\rvert \leq \frac{2M_0}{c}.\ee
				That is to say,
				$$\mathrm{e}^{-\frac{2M_0}{c}} < \frac{c+u}{c-u} < \mathrm{e}^{\frac{2M_0}{c}},     $$
				which implies  that $\lvert u \rvert$ is uniformly bounded away from $c$.\end{proof} 
			
			Second, we confirm that the mass-energy density will keep the positivity property. 
			\begin{lemma}\label{lemma5.2}
				For the $C^1$ solution of the  Cauchy problem $\eqref{RE}$-$\eqref{P}$ with \eqref{initial1}, under the Assumption \ref{p-assumption},  $\rho > 0$.
			\end{lemma}
			
			\begin{proof}  According to \eqref{Pprimeformularel}, one has
				
				\[ \rho= \big( \frac{c}{k}\hspace{.5mm} \text{\text{Tan}}\big( \frac{(w-z)(\g-1)}{4c \sqrt{\g}} \big) \big)^{\frac{2}{\g-1}} := F(w-z).      \]Notice that $F(0)=0$. Denote $\theta = w-z$. We can then rewrite
				
				\[ \lambda_1 - \lambda_2 = - \frac{2 \sqrt{P'(\rho)}\hspace{.5mm}(1-u^2/c^2)}{1-u^2 P'(\rho)/c^4} = - \frac{2 \sqrt{P'(F(\theta))} (1-u^2/c^2)}{1- u^2 P'(F(\theta))/c^4}=g(\theta, u).        \]Notice then that $g(0,u) =0$ and \[  - \frac{2 \sqrt{P'(F(\theta))} (1-u^2/c^2)}{1- u^2 P'(F(\theta))/c^4}  = \frac{\partial g(\theta_1,u)}{\partial u}\theta,          \]where $\theta_1$ is between $0$ and $\theta$. Thus,
				
				\[ (w-z)_t + \lambda_2 (w-z)_x = \frac{\partial g (\theta_1, u)}{\partial u }(w-z) z_x.      \]Let $x^2 = x^2(t,x^2_0)$ denote the the forward characteristic curve starting from the point $(0,x^2_0)$. Integrating along this forward characteristic over $[0,t]$, one can get
				\[ (w-z)(t,x) = (w_0-z_0)(x^2_0) \hspace{.5mm} \text{exp}\big(  \int_0^t \frac{\partial g (\theta_1, u)(\sigma, x^2(\sigma,x^2_0))}{\partial \theta} z_x(\sigma,x^2(\sigma,x^2_0)) \hspace{.5mm} \text{d}\sigma \big),       \]
				which, along with Assumption \ref{p-assumption}, implies the desired conclusion.\end{proof}
			
			Next, we show that the local sound speed $\sqrt{P'(\rho)} $ is also less than the light speed $c$.
			\begin{lemma}\label{light}
				For the $C^1$ solution of the Cauchy problem $\eqref{RE}$-$\eqref{P}$ with \eqref{initial1}, there holds $\sqrt{P'(\rho)} < c$.
			\end{lemma}\begin{proof} According to \eqref{Pprimeformularel}, Assumption \ref{p-assumption}  and $w-z \leq w_{max} -z_{min}$, one gets
				\be \sqrt{P'(\rho)} \leq c\hspace{.5mm} \sqrt{\g} \hspace{.5mm} \text{\text{Tan}}\Big(\frac{(w_{max}-z_{min})(\g-1)}{4c \sqrt{\g}} \Big)< c,       \ee 
				which can be  directly translated to an upper bound on the density $\rho$,
				\be \rho < c^{\frac{2}{\g-1}}k^{-\frac{2}{\g-1}}\gamma^{-\frac{1}{\g-1}}. \ee \end{proof} 
			
			Lemmas $5.1$-$5.3$ show us  that both $\lambda_1, \lambda_2 \in \mathbb{R}$ and  $\lambda_1 < \lambda_2$, which implies that  the system \eqref{RE}-\eqref{P} is  strictly hyperbolic.
			\vspace{3mm}Now we note down the Riccati equations satisfied by $\xi$ and $\zeta$ defined in Section \ref{subsection2.1}. 
			
			\begin{lemma}\label{5.4} For the $C^1$ solution of the  Cauchy problem $\eqref{RE}-\eqref{P}$ with \eqref{initial1}, under the Assumption \ref{p-assumption}, 		the following Riccati ODEs  hold:
				
				\be \xi^{\prime} = -\big(\e^{-{h}_1}{\lambda}_{1 z}\big)\xi^2, \hspace{2mm} \zeta^{\backprime}= -\big(\e^{-{h}_2}{\lambda}_{2 w}\big)\zeta^2. \ee 
				Moreover, let $x^i(t,x^i_0)$ $(i=1,2)$ be two characteristic curves  starting from $(0,x^i_0)$. One has 
				\[ \frac{1}{\xi(t,x^1(t,x^1_0))} = \frac{1}{\xi(0,x^1_0) } + \int_0^t \big(\mathrm{e}^{-h_1}\lambda_{1z}\big) (\sigma,x^1(\sigma,x^1_0))\hspace{.5mm} \text{d}\sigma.   \] 
				\[ \frac{1}{\zeta(t,x^2(t,x^2_0))} = \frac{1}{\zeta(0,x^2_0) } + \int_0^t \big(\mathrm{e}^{-h_2}\lambda_{2w}\big) (\sigma,x^2(\sigma,x^2_0))\hspace{.5mm} \text{d}\sigma.   \]

			\end{lemma}		
			The proof is identical to that of Lemma \ref{lemma41}. We omit the details.
			In what follows, it turns out that there is a clear distinction in the proof between the cases $\gamma \geq 3$ and the \textit{physical} range $1 \leq \gamma < 3$. Because of that, we will lay down the proof for each of those two cases in separate subsections.

			\subsection{Proof for the case $\gamma \geq 3$ of Theorem \ref{1RE}}\label{subsection5.2}
			\qquad  Instead of working with \eqref{RE}-\eqref{P}, here and throughout we will rephrase the problem entirely in the language of Riemann invariants $w$ and $z$, i.e. system \eqref{REODE}.
			To that end, we must first focus our attention on rewriting $\lambda_1$ and $\lambda_2$ as functions of $w$ and $z$ instead of $\rho$ and $u$.
			
			\begin{lemma}\label{lemma5.5}
				Define the following functions
				
				\begin{gather}\label{fg}
				f(w,z):= \frac{w+z}{c} + \ln\Big( \frac{1- \sqrt{\gamma}\hspace{1mm} \text{Tan}\big(\frac{(w-z)(\gamma-1)}{4 c \sqrt{\gamma}}\big)}{1+\sqrt{\gamma}\hspace{1mm}\text{Tan}\big( \frac{(w-z)(\gamma-1)}{4 c \sqrt{\gamma}}\big) } \Big), \\[3pt]
				g(w,z):=  \frac{w+z}{c} + \ln\Big( \frac{1+ \sqrt{\gamma}\hspace{1mm} \text{Tan}\big( \frac{(w-z)(\gamma-1)}{4 c \sqrt{\gamma}}\big)}{1-\sqrt{\gamma}\hspace{1mm}\text{Tan}\big( \frac{(w-z)(\gamma-1)}{4 c \sqrt{\gamma}}\big) } \Big).
				\end{gather}  For the $C^1$ solution of the  Cauchy problem $\eqref{RE}-\eqref{P}$ with \eqref{initial1}, under the Assumption \ref{p-assumption},   there holds
				
				\be \lambda_1(w,z) = c\hspace{.5mm} \frac{\e^f -1}{\e^f +1}, \hspace{2mm} \lambda_2(w,z) = c\hspace{.5mm} \frac{\e^g-1}{\e^g +1}. \ee
			\end{lemma}
			
			\begin{proof}
				We notice that, the functions $\lambda_1, \lambda_2$, written in terms of $\rho$ and $u$, are reminiscent of (in fact identical to) the relativistic addition formulae for $u$ and $\sqrt{P'(\rho)}$. In particular, 
				\begin{gather}
				\ln\big(\frac{c+\lambda_1}{c-\lambda_1}\big) = \ln \big( \frac{c+u}{c-u}\big) + \ln\big( \frac{c-\sqrt{P'}}{c+\sqrt{P'}} \big)=\tilde{f}(\rho,u)=f(w,z), \\
				\ln\big(\frac{c+\lambda_2}{c-\lambda_2}\big) = \ln \big( \frac{c+u}{c-u}\big) + \ln\big( \frac{c+\sqrt{P'}}{c-\sqrt{P'}} \big)=\tilde{g}(\rho,u)=g(w,z), \label{tafg}
				\end{gather}
				which, along with \eqref{Pprimeformularel}  and solving for $\lambda_1, \lambda_2$,  yields the desired relations.
			\end{proof}

			Now we are ready to give the  proof for the case $\gamma \geq 3$ of Theorem \ref{1RE}.
			
			\begin{proof} Assume now, without loss of generality\footnote{If instead $w_0'(x) <0$ the proof is precisely the same after relabelling the corresponding variables.}, that there exists $x \in \mathbb{R}$ such that $z_0'(x) <0$. 			 
				According to Lemma \ref{5.4}, what we need to show is just  the divergence of the integral \[\int_0^\infty (\e^{-h_1}\lambda_{1z})(\sigma,x^1(\sigma,x^1_0))\hspace{1mm} \text{d} \sigma.\]  For this purpose, we divide the the rest of the proof into two steps. 
				
				\textbf{Step 1}: The detailed formula of $\e^{-h_1}\lambda_{1z}$. 
				Introduce, for convenience, the notation 
				\[ Y= \frac{(w-z)(\gamma-1)}{4 \hspace{.5mm} c \hspace{.5mm} \sqrt{\g}}.   \]
				It follows from direct calculations that 
				\be      \lambda_{1w} = \frac{\e^{\frac{w+z}{c}}\big(2-2\g + (1+\g)\text{Cos}(2 \hspace{.5mm} Y)\big) \text{Sec}(Y)^2      }{\big(1+\e^{\frac{w+z}{c}} -(\e^{\frac{w+z}{c}}-1)\hspace{.5mm} \sqrt{\gamma}\hspace{.5mm} \text{\text{Tan}}(Y) \big)^2},\ee 
				and 
				\be \lambda_1 - \lambda_2 =\frac{ 8 \hspace{.5mm}c\hspace{.5mm} \e^{\frac{w+z}{c}} \sqrt{\g}\hspace{.5mm} \text{\text{Tan}}(Y)}{-(1+\e^{\frac{w+z}{c}})^2 +(\e^{\frac{w+z}{c}}-1)^2\hspace{.5mm} \g\hspace{.5mm} \text{\text{Tan}}(Y)^2}. \ee Upon simplification, we have 
				
				\be \label{complicatedexpression} \frac{\lambda_{1w}}{\lambda_1 - \lambda_2} = \frac{\big(2-2\g + (1+\g) \text{Cos}(2Y) \big)\big( (\e^{\frac{w+z}{c}} -1)\sqrt{\g}+ (1+\e^{\frac{w+z}{c}})\text{Cot}(Y)\big) \text{Cosec}(Y) \text{Sec}(Y)}{8 \hspace{.5mm} c \hspace{.5mm}\gamma \hspace{.5mm}  (\e^{\frac{w+z}{c}}-1) - 8 \hspace{.5mm} c\hspace{.5mm} \sqrt{\gamma} \hspace{.5mm}(1+\e^{\frac{w+z}{c}})\text{Cot}(Y)    }. \ee 
				It should be pointed out that the complicated expression \eqref{complicatedexpression} can be explicitly integrated with respect to $w$, which provides  us with an explicit form for the function $h_1$ satisfying \eqref{h1h2def}:			\be  \begin{split} h =& \frac{3\gamma-1}{2\gamma-2}\hspace{.5mm} \text{ln}(\text{Cos}(Y)) + \frac{\gamma-3}{2\gamma-2} \hspace{.5mm}\text{ln}(\text{Sin}(Y)) + \frac{w-z}{2c} \\
					&- \text{ln} \Big( \big(1+   \e^{\frac{w+z}{c}}\big)\hspace{.5mm}\text{Cos}(Y) - \big(-1+ \e^{\frac{w+z}{c}}\big) \hspace{.5mm} \text{Sin}(Y)            \Big). \end{split}\ee
				Befo
				
				Also we notice that  the term  \[ \big(1+   \e^{\frac{w+z}{c}}\big)\hspace{.5mm}\text{Cos}(Y) - \big(-1+ \e^{\frac{w+z}{c}}\big) \hspace{.5mm} \text{Sin}(Y) \] contained inside the $\text{ln}$-function is positive, as $\text{\text{Tan}}(Y) <\frac{1}{\sqrt{\g}}<1$ because of \eqref{Pprimeformularel} and Assumption \ref{p-assumption}.
				
				It follows from the direct calculation that 
				\be \label{lambda1z}\lambda_{1z}= \frac{ \e^{\frac{w+z}{c}}\hspace{.5mm}(1+\g)\hspace{.5mm}\text{Cos}(2Y)\hspace{.5mm}\text{Sec}(Y)^2}{\big(1+\e^{\frac{w+z}{c}} -(\e^{\frac{w+z}{c}}-1)\hspace{.5mm}\sqrt{\g}\hspace{.5mm}\text{\text{Tan}}(Y)\big)^2}.
				\ee It is, at this point, more useful to actually rewrite $\e^{-h_1}\lambda_{1z}$ in terms of the original $(\rho,u)$-coordinates. Define for convenience
				
				\[y = \frac{k\rho^{(\g-1)/2}}{c},\] then the following formula holds:
				
				\be \label{explicitformula1}\e^{-h_1}\lambda_{1z}= \frac{c \hspace{.5mm} \e^{-\frac{2\sqrt{\gamma} \text{Arctan}(y)}{\g-1} }(c+u)(\g+1) \big(\frac{y}{\sqrt{y^2+1}}\big)^{\frac{3-\g}{2\g-2}} (1+y^2)^{\frac{\g+1}{4\g-4}}(1-y^2)        }{2c^2 - 2u\sqrt{P'(\rho)}}>0.           \ee 
				
				\textbf{Step 2}: the uniform  lower bound of  $\e^{-h_1}\lambda_{1z}$.  We analyze the above formulae of $\e^{-h_1}\lambda_{1z}$  	term by term: 		
				\begin{itemize}
					\item the term $c \hspace{.5mm} \e^{-\frac{2\sqrt{\gamma} \text{Arctan}(y)}{\g-1} }$ can be  bounded below by $c \hspace{.5mm} \mathrm{e}^{-\frac{\pi \sqrt{\gamma}}{\gamma-1}}$;

					\item the term $\frac{(c+u)(\gamma+1)}{2c^2 -2 u\sqrt{P'(\rho)}}$  can be uniformly  bounded below  by a positive constant $C_0$ depending only on the initial data and $c$, because of Lemmas \ref{lemma5.1} and \ref{light};
					
					\item the term $\big(\frac{y}{\sqrt{y^2+1}}\big)^{\frac{3-\g}{2\g-2}}(1+y^2)^{\frac{\g+1}{4\g-4}}$  is bounded below by $1$ for all $\g \geq 3$;

					\item the term $1-y^2$  can be bounded below by 
					$1 -\frac{1}{\g}$,
					as $\sqrt{P'(\rho)}<c$ implies \be \label{crucial} y<\g^{-\frac{1}{2}}, \ee
					which will play an important role in  dealing with the physical case $1\leq \g \leq 3$ later.

				\end{itemize}
				
				Then, according to the solution's formula shown in Lemma \ref{5.4} and the uniform lower bound of $e^{-h_1}\lambda_{1z}$ obtained above,  our theorem thus follows for the case $\g \geq 3$.
				
			\end{proof}
			
			\begin{remark}\label{remark4}
				It is easy to see by the above discussion that the term which will ultimately dictate the divergence, or the (hopefully not to be encountered!) convergence  of the integral of $e^{-h}\lambda_{1z}$ in time is  \be \label{Yquantity} \big( \frac{k \rho^{(\gamma-1)/2}/c}{\sqrt{1 + \frac{k^2  \rho^{\gamma-1}}{c^2}}}\big)^{\frac{3-\gamma}{2\gamma-2}}\big(1+ \frac{k^2 \rho^{\gamma-1}}{c^2}\big)^{\frac{\gamma+1}{4\gamma-4}} := \mathcal{Y}.\ee 
			\end{remark}

			\subsection{Proof for the case  $1\leq \g< 3$ of Theorem \ref{1RE}}
			
			\qquad Our starting point is Remark \ref{remark4} from Subsection \ref{subsection5.2}. We observe  that the behavior of $\mathcal{Y}$ ultimately dictates whether or not the integral $\int_0^\infty \big(\e^{-h_1} \lambda_{1z}\big)(\sigma,x^1(\sigma,x^1_0)) \hspace{0.5mm} \text{d}\sigma$ diverges. In particular, it becomes clear that in order to prove the Theorem \ref{1RE} for  $1\leq \g < 3$,  we are required to give a proper time-dependent lower bound on $\mathcal{Y}$, a function itself of the density $\rho$, strong enough such that $\int_0^\infty \mathcal{Y}(\sigma, x^1(\sigma))\hspace{.5mm} \text{d}\sigma = \infty$.

			Before showing  the detailed proof, we first give  a clear outline of our strategy:
			\begin{itemize} 
			\item  (\textrm{1})\quad We rewrite the equations \eqref{REODE} in the form of  the difference $(w-z)$ of the Riemann invariants:
			\be  \label{wwzz} (w-z)'  = (w-z)_t + \lambda_1 (w-z)_x = (\lambda_1 - \lambda_2)w_x, \ee 
			which implies that, 
	\be \label{startingpoint}(w-z)' = \frac{\lambda_1- \lambda_2}{\mathrm{e}^{h_2}}\zeta   \geq - \frac{\lambda_2- \lambda_1}{\mathrm{e}^{h_2}} Q_2,  \ee
	under the assumption that $\zeta=\e^{h_2} w_x$ has a uniform upper bound $Q_2$ independent of the time. Actually, such kind a bound is established in Subsection 5.3.1;

				\item (\textrm{2})\quad In order to introduce a suitable ODEs inequality for $\mathcal{Y}$, we first obtain an   ODEs inequality for $y$ from \eqref{startingpoint}, which requires us to      rewrite     the quantities $w-z$ and $\frac{\lambda_2- \lambda_1}{\mathrm{e}^{h_2}}$    explicitly in terms of $\rho, u$.  To this end, there is a crucial observation  that should be pointed out, that   in $\frac{\lambda_2- \lambda_1}{\mathrm{e}^{h_2}}$,  all the terms involving $u$ have a uniform upper and lower bound.  This can be seen in Section 5.3.2;

				\item (\textrm{3})\quad Based on the analysis on $y$ and the relation between $y$ and $\mathcal{Y}$, we successfully  introduce a proper ODEs inequality for $\mathcal{Y}$, which  could  effectively  control the behaviour of the density with respect to time. This can be seen in Section 5.3.3;
								
				\item (\textrm{4})\quad Finally, we show that the estimates obtained above  are indeed sufficiently strong so that the desired quantity \eqref{Yquantity} has a divergent integral over time along the characteristic curve. Then we can obtain the desired conclusion stated in Theorem \ref{1RE}.
			\end{itemize}With that in mind, let us begin to show the detailed proof step by step.

			\subsubsection{Upper bound of the weighted gradients of the Riemann invariants}
			
			\qquad Now we need to give the 	upper bound of the weighted gradients of the Riemann invariants.		
			
			\begin{lemma}\label{Q1Q2lemma} Define the non-negative constants
				\be Q_1 := max\begin{Bmatrix} 0, \sup_x \xi(x,0) \end{Bmatrix}, \hspace{2mm} Q_2 := max \begin{Bmatrix} 0, \sup_x \zeta(x,0)\end{Bmatrix}.   \ee For the $C^1$ solution of the system $\eqref{RE}$-$\eqref{P}$,  under Assumption \ref{p-assumption}, 	one has
				\be \xi(x,t) \leq Q_1 , \hspace{2mm} \zeta(x,t) \leq Q_2. \ee \end{lemma}

	\begin{proof}  The key to the proof lies in showing that $\e^{-h_1}\lambda_{1z}$ and $\e^{-h_2}\lambda_{2w}$ are non-negative. For this purpose, according to \eqref{lambda1z} and the formula 
				\[    \lambda_{2w} = \frac{    \e^{\frac{w+z}{c}}\hspace{.5mm}        (1+\g)\hspace{.5mm} \text{Cos}(2Y)\hspace{.5mm} \text{Sec}(Y)^2}{\big(1+\e^{\frac{w+z}{c}}  + (\e^{\frac{w+z}{c}}-1 )\hspace{.5mm}  \sqrt{\g}\hspace{1mm} \text{\text{Tan}}(Y)\big)^2  },        \] one gets that what we need is $\text{Cos}(2Y) > 0$, which, actually, can be obtained quickly from \eqref{crucial}, and   the following  formulas
\[Y= \frac{(w-z)(\g-1)}{4 \hspace{.5mm} c \hspace{.5mm} \sqrt{\g}}= \text{Arctan}\Big(\frac{k \rho^{(\g-1)/2}}{c}\Big)\quad \text{and} \quad  \text{Cos}(2\text{Arctan}(x)) = \frac{1-x^2}{1+x^2}.    \]

\par \noindent Thus,  the conclusion of this lemma follows  from the  Riccati ODEs established in Lemma \ref{5.4}.	
				
			\end{proof}		
			
		It should be pointed out that, so far, \eqref{startingpoint} has been proved.

			\subsubsection{Establishing ODE inequality of $y$}
			
			 \qquad As mentioned before, now  we need to rewrite     the quantities $w-z$ and $\frac{\lambda_2- \lambda_1}{\mathrm{e}^{h_2}}$    explicitly in terms of $\rho, u$.			
		First, according to \eqref{Pprimeformularel} and \eqref{fg}, one has
			\begin{align} 
			f= \ln\Big(\frac{c+u}{c-u}\Big) + \ln \big( \frac{1- \sqrt{\gamma} \frac{ k \rho^{(\gamma-1)/2}}{c}      }{1+ \sqrt{\gamma} \frac{ k \rho^{(\gamma-1)/2}}{c} } \big), \\ g= \ln\Big(\frac{c+u}{c-u}\Big) + \ln \big( \frac{1+ \sqrt{\gamma} \frac{ k \rho^{(\gamma-1)/2}}{c}      }{1- \sqrt{\gamma} \frac{ k \rho^{(\gamma-1)/2}}{c} } \big).
			\end{align}
			Second, from  \eqref{P} and \eqref{lambda12}, one can get 
\begin{align}
			\lambda_1 = \frac{      c^2(u -k \hspace{.5mm} \sqrt{\gamma} \hspace{.5mm} \rho^{(\gamma-1)/2})}{c^2  - k\hspace{.5mm} u \sqrt{\gamma} \hspace{.5mm}\rho^{(\gamma-1)/2}}, \\ \lambda_2 =\frac{      c^2(u +k \hspace{.5mm} \sqrt{\gamma} \hspace{.5mm} \rho^{(\gamma-1)/2})}{c^2  + k\hspace{.5mm} u \sqrt{\gamma} \hspace{.5mm}\rho^{(\gamma-1)/2}}, 
			\end{align}whence the following simplified expression for $\lambda_2 - \lambda_1$ follows:
			
			\be \lambda_2 - \lambda_1     = \frac{2c^2\hspace{.5mm} k\hspace{.5mm} (c^2-u^2)\hspace{.5mm} \sqrt{\gamma}\hspace{.5mm} \rho^{(\gamma-1)/2}  }{c^4 - k^2\hspace{.5mm} u^2\hspace{.5mm} \rho^{\gamma-1}}.
			\ee
			 Moreover, the following explicit forms for $h_1$ and $h_2$ can be obtained:
			
			\begin{equation*}\begin{split}
			h_1=& \frac{3\gamma-1}{2\gamma-2}\ln(\text{Cos}(Y)) + \frac{\gamma-3}{2\gamma-2} \ln(\text{Sin}(Y)) + \frac{w-z}{2c} \\
			& - \ln \Big(\big(1+   \e^{\frac{w+z}{c}}\big)\text{Cos}(Y) - \big(-1+ \e^{\frac{w+z}{c}}\big)\hspace{.5mm} \sqrt{\gamma} \hspace{.5mm}\text{Sin}(Y)   \Big), \\
			h_2= &\frac{3\gamma-1}{2\gamma-2}\ln(\text{Cos}(Y)) + \frac{\gamma-3}{2\gamma-2} \ln(\text{Sin}(Y)) + \frac{w-z}{2c}  \\
			& -\ln \Big(\big(\e^{\frac{2w}{c}}+   \e^{\frac{w-z}{c}}\big)\text{Cos}(Y) + \big(\e^{\frac{2w}{c}}-  \e^{\frac{w-z}{c}}\big)\hspace{.5mm}\sqrt{\gamma} \hspace{.5mm}\text{Sin}(Y)   \Big).
			\end{split}
			\end{equation*}
			Here, as always, we denote $Y= \text{Arctan}\big(\frac{k \rho^{(\gamma-1)/2}}{c} \big)= \frac{(w-z)(\gamma-1)}{4 c \sqrt{\gamma}}$.

			Now we are ready to  develop one    ODEs inequality for $y$ from \eqref{startingpoint}.			
			\begin{lemma}\label{lemma5.7}
			For the $C^1$ solution of the  Cauchy problem $\eqref{RE}-\eqref{P}$ with \eqref{initial1}, under Assumption \ref{p-assumption}, 			
			there holds \be \label{yode} y' \geq -C_{g}\hspace{.5mm}  y^{\frac{\g+1}{2\g-2}}\hspace{.5mm} (1+y^2)^{3/2}, \ee
				for some positive constant $C_g$ independent of the time. 
			\end{lemma}  \begin{proof} First, it follows from the  direct  computations that the following simplified form for  $\frac{\lambda_2 - \lambda_1}{ \e^{h_2}}$ holds: 
			\be \label{5.31} \frac{ \lambda_2 - \lambda_1}{\e^{h_2}}= \hspace{.5mm} \frac{4 c \hspace{.5mm} \e^{ \frac{2\hspace{.5mm} \sqrt{\gamma}\hspace{.5mm} \text{Arctan}(y) }{\gamma-1}    } \hspace{.5mm}  \hspace{.5mm} (c+u)\hspace{.5mm}  c\hspace{.5mm} \sqrt{\gamma}\hspace{.5mm}  y\hspace{.5mm}  \big(\frac{y}{\sqrt{y^2+1}}\big)^\frac{3-\gamma}{2\gamma-2}  (1+y^2)^\frac{\gamma+1}{4\gamma-4}                  }{           c^2- u \sqrt{P'(\rho)}},                     \ee where $y=\frac{k \rho^{(\gamma-1)/2}}{c}$.    
			
			Second, according to Lemmas \ref{lemma5.1}-\ref{light},   one can obtain that 
\begin{equation}\label{cc1}
\begin{split}
C^{-1}_g\leq &  \e^{ \frac{2\hspace{.5mm} \sqrt{\gamma}\hspace{.5mm} \text{Arctan}(y) }{\gamma-1}    }\leq C_g, \quad \text{and} \quad C^{-1}_g\leq  \frac{c+u}{c^2-u \sqrt{P'(\rho)}}\leq C_g,	
\end{split}		
\end{equation}
for some positive constant $C_g$ independent of the time, which, along with \eqref{startingpoint}, implies 
\be \label{5.34} (w-z)' \geq - C_{g}\hspace{.5mm}   y\hspace{.5mm}  \big(\frac{y}{\sqrt{y^2+1}}\big)^\frac{3-\gamma}{2\gamma-2}  (1+y^2)^\frac{\gamma+1}{4\gamma-4} = - C_{g}\hspace{.5mm}  y^{\frac{\gamma+1}{2\gamma-2}}\hspace{.5mm}  (1+y^2)^{\frac{1}{2}}.   \ee

Notice that 
$$w-z= \frac{4 \hspace{.5mm}  c \hspace{.5mm}  \sqrt{\gamma}}{\gamma-1}\text{Arctan}(y),\quad \text{and} \quad (w-z)' = C_{g} \frac{y'}{y^2+1},   $$
so that  we can  rewrite \eqref{5.34} as  \eqref{yode}.
\end{proof}

			\subsubsection{Establishing one ODE inequality of $\mathcal{Y}$}

			\qquad 	We now recall that what determines the convergence/divergence of the integral of $\e^{-h_1}\lambda_{1z}$ in time is the variable $\mathcal{Y}$ from \eqref{Yquantity}, which can be rewritten as 
			
			\be \mathcal{Y}=\big(\frac{y}{\sqrt{y^2+1}}\big)^\frac{3-\g}{2\g-2} (1+y^2)^{\frac{\g+1}{4\g-4}}. \ee 
			
			At this point, without loss of generality\footnote{If instead $w_0'(x) <0$ the proof is precisely the same after relabelling the corresponding variables.}, that there exists $x \in \mathbb{R}$ such that $z_0'(x) <0$. 			 
				According to Lemma \ref{5.4}, what we need to show is just  the divergence of the integral \[\int_0^\infty (\e^{-h_1}\lambda_{1z})(\sigma,x^1(\sigma,x^1_0))\hspace{1mm} \text{d} \sigma.\] 			
	Then, according to fact on $\mathcal{Y}$ mentioned above, our task therefore reduces to showing that $$\int_0^\infty \Big(\big(\frac{y}{\sqrt{y^2+1}}\big)^\frac{3-\g}{2\g-2} (1+y^2)^{\frac{\g+1}{4\g-4}}\Big)(t,x^1(t,x^1_0))\text{d}t$$
	 diverges, based on  Lemma \ref{lemma5.7}.  
	 
	 However, the explicit solution of the differential equation 
	 $$
	 y' =-C_{g}\hspace{.5mm}  y^{\frac{\g+1}{2\g-2}}\hspace{.5mm} (1+y^2)^{3/2}	 
	 $$
 is very hard to handle, as it involves hypergeometric functions. We instead adopt an indirect approach and look at a transport equation for the variable $\mathcal{Y}$ itself:
			
			\be \mathcal{Y}= \big(\frac{y}{\sqrt{y^2+1}}\big)^\frac{3-\g}{2\g-2} (1+y^2)^{\frac{\g+1}{4\g-4}}= y^\frac{3-\g}{2\g-2} (1+y^2)^\frac{1}{2} . \ee
			It turns out that, one can obtain the following  Riccati-type inequality for $\mathcal{Y}$:		
			\begin{lemma}\label{lemma5.8}
		For the $C^1$ solution of the  Cauchy problem $\eqref{RE}-\eqref{P}$ with \eqref{initial1}, under the Assumption \ref{p-assumption},  there holds 
		$$\mathcal{Y}' \geq - C_{g}\hspace{.5mm} \mathcal{Y}^2$$ for some universal constant $C_{g}$ independent of the time.
			\end{lemma}
			\begin{proof} First, it follows from the direct calculation that 
\be \mathcal{Y}' = \frac{y^\frac{5-3\g}{2\g-2}\big((\g+1)y^2 +(3-\g)\big) y' }{(2\g-2) \sqrt{y^2+1}}, \ee
which, along with \eqref{yode}, implies that  
\be \mathcal{Y}'\geq -C_{g}\hspace{.5mm}  y^\frac{6-2\g}{2\g -2} \hspace{.5mm} (1+y^2)\hspace{.5mm} \big((\g+1)y^2 +(3-\g)\big),    \ee 
Actually, the above ODE inequality  can equivalently be rewritten as
				\be \label{finalode}\mathcal{Y}' \geq - C_{g}\hspace{.5mm}\mathcal{Y}^2 \big((\g+1)y^2 +(3-\g)\big). \ee At this point it is important to observe the following:
				\vspace{3mm}
				
			Second, it follows from the fact   $\sqrt{P'(\rho)}< c$  that 
			$$y= \frac{k \rho^{(\g-1)/2}}{c} < \g^{-\frac{1}{2}}.$$
			 Therefore \[(\g+1) y^2 + (3-\g) \leq \frac{\gamma+1}{\gamma} + (3-\g),\]
			 which, along with \eqref{finalode}, implies that 
\be \mathcal{Y}' \geq - C_{g}\hspace{.5mm}  \mathcal{Y}^2,\ee
 for some universal constant $C_{g}$ independent of the time.

\end{proof} 
							
			\subsubsection{Mass-density's lower bound estimates and formation of singularity}		
				
		\qquad  Now, based on the conclusions obtained in Sections 5.3.1-5.3.3, we are ready to finish the proof of Theorem \ref{1RE}.
			
			Actually, from Lemma \ref{lemma5.8}, one can obtain that 
			\be \label{finally} \mathcal{Y}(t) \geq \frac{1}{C_1 + C_2t}, \ee 
	for some universal constants $C_{i}$  $(i=1,2)$ independent of the time.			
			
			From \eqref{explicitformula1}, one can get 
			\be \label{explicitformula1final}\e^{-h_1}\lambda_{1z}= \frac{c \hspace{.5mm} \e^{-\frac{2\sqrt{\gamma} \text{Arctan}(y)}{\g-1} }(c+u)(\g+1) (1-y^2)        }{2c^2 - 2u\sqrt{P'(\rho)}}\mathcal{Y}=\mathcal{H}\mathcal{Y}>0.           \ee 			
	From the analysis shown in Step 2 of  the proof for the case $\gamma \geq 3$ of Theorem \ref{1RE}	in Section 5.2,  one obtains 
	$$
	C^{-1}_g \leq \mathcal{H} \leq C_g, 
	$$	
	for some universal constant $C_g$ independent of the time.	Also, we know that  $\mathcal{Y}(0)$ is positive.
	
	Then, finally, according to the solution's formula shown in Lemma \ref{5.4}, the desired conclusion stated in Theorem \ref{1RE} for $1<\gamma<3$ has been proved.

	\begin{remark}
		Notice that $\mathcal{Y} = C \rho^{\frac{3-\gamma}{4}} \sqrt{1+\frac{k^2\rho^{\gamma-1} }{c^2}}< 2 C \rho^{\frac{3-\gamma}{4}},$ because of the upper bound on $\rho$. Together with \eqref{finally}, we obtain the same bound for $\rho$ as in the classical case:
		
		\[  \rho(t,x) = O(1+t)^{-\frac{4}{3-\gamma}}.    \]
	\end{remark}
				
			\section{Formation of singularities for the (3+1)-dimensional RE}\label{section6}
			\qquad   Let $d=3$ in \eqref{RE} and \eqref{initial1}. In this section we will do some studies on the formation of singularities for the $(3+1)$-dimensional relativistic fluids. 
			\subsection{Proof of Theorem \ref{coo2}}

			\qquad In this subsection, we always assume that   $(\rho, u)(t,x)$ is the  regular solution  in $[0, T_m]\times\mathbb{R}^3$ defined  in Definition \ref{d1} to the Cauchy problem  \eqref{RE}-\eqref{P} and \eqref{initial1}.		
			Before showing the  singularity formation caused by the    so-called  \textit{isolated mass group}, defined in Definition \ref{local}, we first consider the time evolution of the vacuum domain.  For this purpose, next we introduce one physical  quantity, namely the \textit{particle number}.
			
			\begin{definition}\label{particlenumber}
				Define the particle number $n$ by 
				\be \label{equa}
				n(\rho)=n(1)\exp\Big(\int_1^\rho \frac{\text{d}\sigma}{\sigma + \frac{P(\sigma)}{c^2}} \Big).
				\ee
			\end{definition}
			
			In the following lemma, we will show that the evolution of the vacuum still can be tracked by the particle path defined by \eqref{gobn}.
			\begin{lemma}\label{vacuumequavalence}
				\be \label{equalence}
				n(\rho)=0 \quad  \text{if and only if }\quad \rho=0.
				\ee
				Moreover, one has 
				\be \label{nformula}
				\frac{n}{\sqrt{1-u^2/c^2}}(t, x(t;x_0))=\frac{n}{\sqrt{1-u^2/c^2}}(0, x_0)	\exp\Big(\int_0^t -\text{div}u(\sigma, x(\sigma;x_0))\text{d}\sigma\Big),\ee
				where   $x(t;x_0)$ is the particle path starting from $(0,x_0)$.  Also, if $\rho(0, x_0)=0$, then one gets
				$$
				\rho(0, x_0)=n(0, x_0)=n(t, x(t;x_0))=\rho(t, x(t;x_0))=0 \quad \text{for} \quad t\in [0,T_m].$$
			\end{lemma}
			\begin{proof}
				First, the  equivalence relation \eqref{equalence} follows from the facts  $n(0)=0$ and $\frac{\text{d}n}{\text{d}\rho}>0$.

				Second,  according to the system \eqref{RE}, one can obtain that $n$ satisfies the following equation
				\begin{equation}\label{RE-particle}
				\displaystyle
				\big( \frac{n}{\sqrt{1-u^2/c^2}} \big)_t  + \text{div}\big(\frac{nu}{\sqrt{1-u^2/c^2}} \big) = 0, 
				\end{equation}
				which implies the formula \eqref{nformula}.
			\end{proof}

			For simplicity, we define the following images of $A_0$, $B_0$, and $B_0 \setminus A_0$, respectively, under the flow map of (\ref{gobn}).
						\begin{definition}[\textbf{Particle path and flow map}]\label{kobn}Let $A(t)$, $B(t)$, $B(t)\setminus A(t)$ be the images of $A_0$, $B_0$, and $B_0 \setminus A_0$, respectively, under the flow map of (\ref{gobn}), i.e.,
				\begin{equation*}
				\begin{split}
				&A(t)=\left\{x(t;x_0)|x_0\in A_0\right\},\\[2pt]
				& B(t)=\left\{x(t;x_0)|x_0\in B_0\right\},\\[2pt]
				&B(t)\setminus A(t)=\left\{x(t;x_0)|x_0\in B_0\setminus  A_0\right\}.
				\end{split}
				\end{equation*}
			\end{definition}
			
			 The following lemma establishes the invariance of  the volume $|A(t)|$ for regular solutions.
			\begin{lemma}
				\label{lemma:3.1}
				Suppose that  the initial data $(\rho_0,u_0)(x)$ have an isolated mass group $(A_0,B_0)$.
				Then, for the   regular solution $(\rho,u)(t,x)$ on $\mathbb{R}^3\times[0,T_m)$ to the Cauchy problem  \eqref{RE}-\eqref{P} and \eqref{initial1}, we have
				$$
				|A(t)|=|A_0|, \quad  t\in [0,T_m).
				$$
			\end{lemma}
			
			\begin{proof}
				From  Lemma \ref{vacuumequavalence} and    the definition of  regular solutions,  one has
				\begin{equation}\label{eq:5.3}
				u_t+u\cdot\nabla u=0, \quad \text{in} \quad   B(t)\setminus A(t).
				\end{equation}
				Therefore, $u$ is invariant along the particle path $x(t;x_0)$ with $x_0\in B_0 \setminus A_0$. 
				
				\vspace{3mm}
				\par \noindent For any $x^1_0,\ x^2_0 \in \partial A_0$, we define
				\begin{equation}\label{gobn1}
				\frac{d}{\text{d}t}x^i(t;x^i_0)=u(x^i(t;x^i_0), t),\quad x^i(0;x^i_0)=x^i_0,\quad \text{for} \quad i=1,2.
				\end{equation}
				Then we  have
				\begin{equation}\label{gobn2}
				\frac{d}{\text{d}t}(x^1(t;x^1_0)-x^2(t;x^2_0))=u(x^1(t;x^1_0), t)-u(x^2(t;x^2_0), t)={\bar u}_0-{\bar u}_0=0,
				\end{equation}
				which implies that
				$$
				|A(t)|=|A_0|, \quad  t\in [0,T_m].
				$$
			\end{proof}
			
			 We point out that, although the volume of $A(t)$ is invariant, the vacuum boundary $\partial A(t)$ travels with constant velocity ${\bar u}_0$. The  following well-known Reynolds transport theorem (c.f. \cite{kong}) is useful.
			\begin{lemma}\label{3.1}
				For any $G(t,x)\in C^1(\mathbb{R}^3\times\mathbb{R}^+) $, one has
				$$
				\frac{d}{\text{d}t}\int_{A(t)}  G(t,x)\text{d}x= \int_{A(t)}  G_t(t,x)\text{d}x+\int_{\partial A(t)} G(t,x)(u(t,x)\cdot {\vec n})\text{d}S,
				$$
				where ${\vec n}$ is the outward unit normal vector to $\partial A(t)$ and $u$ is the velocity of the fluid.
			\end{lemma}

	Based the observations in Lemmas \ref{vacuumequavalence}-\ref{3.1}, now we can obtain the following conservation laws of total energy and  total momentum, and the  invariance of the centroid.	
			\begin{lemma}\label{3.2}
				Suppose that  the initial data $(\rho_0,u_0)(x)$ have  an isolated mass group $(A_0,B_0)$,
				then for the   regular solution $(\rho,u)(t, x)$ on $\mathbb{R}^3\times  [0,T_m)$ to the Cauchy problem \eqref{RE}-\eqref{P} with \eqref{initial1} , we have
				$$m(t)=m(0),\quad \mathbb{P}(t)=\mathbb{P}(0),     \ X^*(t)=X^*(0),\ A(t) \subseteq B_{2R_0+|X^*_0|},       \quad \text{for} \quad t\in [0,T_m).
				$$
			\end{lemma}
			\begin{proof} From $(\ref{RE})_1$ and Lemma \ref{3.1}, direct computation shows
				\begin{equation*}\begin{split}
				\frac{d}{dt}m(t)=&\int_{A(t)} \hat{\rho}_t\  \text{d}x+\int_{\partial A(t)} \hat{\rho} u\cdot {\vec n}\ \text{d}S\\
				=&\int_{A(t)} -\text{div}(\tilde{\rho} u)\  \text{d}x=\int_{\partial A(t)} -\tilde{\rho} u\cdot {\vec n}\ \text{d}S=0,
				\end{split}
				\end{equation*} 
				which implies that $m(t)=m(0)$.
				
				Similarly, one has				
				\begin{equation*}\begin{split}
				\frac{d}{dt}\mathbb{P}(t)=&\int_{A(t)} (\tilde{\rho} u)_t \text{d}x+\int_{\partial B(t)} \tilde{\rho} u (u\cdot n)\text{d}S\\
				=&\int_{A(t)}\Big( -\text{div}(\tilde{\rho} u\otimes u)-\nabla P\Big) \text{d}x
				=\int_{\partial A(t)} \Big(-\tilde{\rho} u\otimes u-PI_3\Big)\cdot {\vec n}\  \text{d}S=0,
				\end{split}
				\end{equation*}
				which implies that $\mathbb{P}(t)=\mathbb{P}(0)=0$.				
				
				Finally, from the definition of $X^*(t)$, $m(t)=m_0$ and $\mathbb{P}(t)=\mathbb{P}(0)=0$, one has
				\begin{equation*}\begin{split}
				\frac{d}{dt}\Big(\int_{A(t)} x \hat{\rho} \text{d}x\Big)=&\int_{A(t)} x \hat{\rho}_t \text{d}x+\int_{\partial A(t)} x \hat{\rho} (u\cdot n)\text{d}S
				\\
				=&-\int_{A(t)} x \text{div}(\tilde{\rho} u) \text{d}x
				= \int_{A(t)} \tilde{\rho u} \text{d}x=\mathbb{P}(0)=0,
				\end{split}
				\end{equation*}
				which means that $X^*(t)=X^*(0)$. Moreover, since the total energy on $A(t)$ is conserved, thus $X^*(0)$ is contained in the closed convex hull
				of $A(t)$
				from $X^*(t)=X^*(0)$, we easily know that $A(t)\subseteq B_{2R_0+|X^*(0)|}$.

			\end{proof}

			\par \noindent 	We are now ready to give the proof of  Theorem \ref{coo2}:\\
			\begin{proof} First, it follows from the  equation $(\ref{RE})_1$ and the  integration by parts that 
				\begin{equation}\label{eq:3.8}
				\frac{d}{\text{d}t}M(t)=\int_{A(t)} \hat{\rho}_t |x|^2 \text{d}x+\int_{\partial A(t)} \hat{\rho} |x|^2(u\cdot n) \text{d}S=2F(t).
				\end{equation}
				Similarly, according to  Lemma \ref{3.1},  the  equation $(\ref{RE})_2$ and integration by parts, one can obtain that 
				\begin{equation}\label{eq:3.9}
				\begin{split}
				\frac{d}{\text{d}t}F(t)=&\int_{A(t)} (\tilde{\rho} u)_t \cdot x \text{d}x+\int_{\partial A(t)} \tilde{\rho} u\cdot x(u\cdot n) \text{d}S\\
				=&\int_{A(t)} \big( -\text{div}(\tilde{\rho} u\otimes u)-\nabla P\big)\cdot x \text{d}x\\
				=&\int_{A(t)} \tilde{\rho}u^2\text{d}x+3\int_{A(t)} P\text{d}x
				=c^2\int_{A(t)} \Big(\frac{1}{c^2}\tilde{\rho}u^2+ \frac{3P}{c^2}\Big)\text{d}x.
				\end{split}
				\end{equation}
				
				\par \noindent By Jensen's inequality, one has 
				\begin{equation}\label{eq:3.9J}
				\begin{split}
				\int_{A(t)} P\text{d}x\geq |A(0)|P(\bar{m}(t)),
				\end{split}
				\end{equation}
				where $\bar{m}(t)=\frac{\int_{A(t)}\rho \text{d}x}{|A(0)|}$. Then from \eqref{eq:3.9}-\eqref{eq:3.9J}, one obtains 
				\begin{equation}\label{eq:3.9K}
				\begin{split}
				\frac{d^2}{dt^2}M(t)\geq 2c^2\int_{A(t)} \Big(\frac{1}{c^2}\tilde{\rho}u^2\Big)\text{d}x+6 |A(0)|P(\bar{m})\equiv 2c^2 N(t).
				\end{split}
				\end{equation} 
				
				Now we consider the following two cases:
				\begin{itemize}
					\item If $\int_{A(t)} \frac{1}{c^2}\tilde{\rho}u^2\text{d}x\geq \frac{1}{2}m(0)$, one gets
					\begin{equation}\label{eq:3.9L}
					\begin{split}
					N(t)\geq \frac{1}{2}m(0);
					\end{split}
					\end{equation}
					\item  If $\int_{A(t)} \frac{1}{c^2}\tilde{\rho}u^2\text{d}x\leq  \frac{1}{2}m(0)$, then according to 
					\begin{equation}\label{eq:3.9M}
					\begin{split}
					m(t)=m(0)=\int_{A(t)} \hat{\rho} \text{d}x\leq \frac{1}{2}m(0)+\int_{A(t)} \rho \text{d}x,
					\end{split}
					\end{equation}
					one has
					\begin{equation}\label{eq:3.9N}
					\begin{split}
					\int_{A(t)} \rho \text{d}x \geq \frac{1}{2}m(0),
					\end{split}
					\end{equation}
					which implies that 
					\begin{equation}\label{eq:3.9O}
					\begin{split}
					N(t)\geq \frac{3}{c^2} |A(0)|P(\bar{m}(0))\equiv D_0m(0)>0,\end{split}
					\end{equation}
					where $D_0=\frac{3}{c^2 m(0)} |A(0)|P(\bar{m}(0))$.\end{itemize}

					Denote $D=2c^2\min\{\frac{1}{2},D_0\}$, one has 
				\begin{equation}\label{eq:3.9n}
				\begin{split}
				\frac{d^2}{dt^2}M(t)\geq Dm(0)>0.
				\end{split}
				\end{equation}
				Integrating (\ref{eq:3.8}) and (\ref{eq:3.9n}) over $[0,t]$, respectively, one can obtain that 
				\begin{equation}\label{eq:3.1733}
				M(t)\geq M_0+2 F_0t+\frac{1}{2}Dm(0) t^2.
				\end{equation}
				According to Lemma \ref{3.2}, it yields
				\begin{equation}\label{eq:3.1833}
				M(t)=\int_{A(t)} \hat{\rho} |x|^2 \text{d}x \leq  R^2_1m_0,
				\end{equation}
				where $R_1=|X^*(0)|+2R_0$.
				
				\vspace{3mm}
				
				\par \noindent 	Combining (\ref{eq:3.1733}) with (\ref{eq:3.1833}), one has
				\begin{equation}\label{time}
				R^2_1m_0\geq M_0+2 F_0t+\frac{1}{2}Dm(0) t^2,
				\end{equation}
				which means that $ T_m <+\infty$.
			\end{proof}

			\subsection{Proof of Theorem \ref{th333}}	
			
			\begin{proof}
				We denote by $ V(t) $ the evolutioned domain that is the image of $ V$ under the flow map, i.e.,
				\begin{equation}\label{zhi}
				V(t)=\{x|x=x(t;x_0), \quad \forall \ x_0\  \in V\},
				\end{equation}
				where $ x(t; x_0)$ is the particle path starting from $(0,x_0)$.
				It follows from Lemma \ref{vacuumequavalence} that the mass-energy density is simply supported along the particle paths, so
				$$ \rho(t,x)=0,\quad \text{when} \quad x\ \in \ V(t).$$
				Thus, via the Definition \ref{d1} for regular solutions, we deduce that 
				\begin{equation}
				\label{eq:1.2guo}
				u_t+u\cdot \nabla u=0, \quad \text{when} \quad x\ \in \ V(t),
				\end{equation}
				which means that $u$ is a constant along the particle path $x(t; x_0)$.
				Then for any $x\in V(t)$, we obtain that 
				$$
				u(t,x)=u_0(x-tu(t,x)),
				$$
				which immediately implies that 
				\begin{equation}
				\label{eq:1.2fan}
				\nabla u(t,x)=\big(\mathbb{I}_3+t\nabla u_0(x-tu(t,x))\big)^{-1}\nabla u_0,\quad \text{for} \quad x\in V(t).
				\end{equation}
				If there is any $\lambda\in  Sp(\nabla u_0) $ satisfying $\lambda<0$, 
				then from (\ref{eq:1.2fan}), it is obvious that the quantity $\nabla u$ will blow up in finte time, i.e., 
				$$
				T_m<+\infty.
				$$
				
			\end{proof}
			
			\subsection{Proof of Theorem \ref{th:2.20}}
			\qquad In this subsection, we simply denote 
			$$
			\int_{\mathbb{R}^3}f\text{d}x=\int f\text{d}x.			$$
		 Now we are ready to prove Theorem \ref{th:2.20}. 
			\begin{proof} Let $T>0$ be any constant, and    $(\rho,u)\in D(T)$. It  follows from the definitions of  $m(t)$, $\mathbb{P}(t)$ $\hat{\rho}$ and $\tilde{\rho}$  that
			\be \begin{split}
				|\mathbb{P}(0)|=& |\mathbb{P}(t)|\leq \|u(t)\|_{L^\infty(\mathbb{R}^3)} \int \tilde{\rho} \text{d}x
				\leq  \|u(t)\|_{L^\infty(\mathbb{R}^3)} \Big(m(0)+\int_{\mathbb{R}}\frac{P}{c^2}\text{d}x\Big).
			\end{split}
			\ee
			Notice that 
			\be \begin{split}
				\int P\text{d}x=& \Big| \int \int_0^\rho P'(\sigma)\text{d}\sigma\text{d}x\Big|
				\leq \Big| \int \|P'(\rho)\|_{L^\infty(\mathbb{R}^3)}\rho \text{d}x\Big|\\
				\leq & c^2\int \rho\text{d}x\leq c^2 m(t)=c^2 m_0.
			\end{split}
			\ee
			
			\par \noindent 	Then one obtains that there exists a positive constant $C_u=\frac{|\mathbb{P}(0)|}{2m(0)}$ such that
			$$
			\|u(t)\|_{L^\infty(\mathbb{R}^3)}\geq C_u  \quad \text{for} \quad t\in [0,T].
			$$
			Thus one obtains     the desired conclusion as shown in Theorem \ref{th:2.20}.
			\end{proof}

			\vspace{3mm}

			\section{Remarks on the  general pressure law of the 1-dimensional case}\label{section7}
			
			\qquad     	\quad Let $d=1$ in \eqref{RE}-\eqref{CE} and \eqref{initial1}.      We revisit in this section the study of the relativistic Euler equations in $1+1$ dimensions. Its main aim is to enlarge the set of pressure laws $P=P(\rho)$ from those  in \eqref{P} to more general ones which  satisfy  Assumptions   \ref{assumption2}-\ref{assumption3}.	As mentioned before, we first show our basic idea for the classical compressible Euler equations, and then give detailed proof for the  relativistic flow.  Throughout this section, we always   assume that  $P(\rho)$ satisfies the  Assumptions   \ref{assumption2}-\ref{assumption3}.

			\subsection{The classical Euler equations}\label{subsection7.1}
			\subsubsection{Notations and relations.}
			\qquad First, the eigenvalues  $\tilde{\lambda}_i$ $(i=1,2)$, directional derivatives $\partial_-$ and $\partial_+$, the characteristic directions $y^1$ and $y^2$, Riemann variables $\tilde{w}$ and $\tilde{z}$,  $\tilde{h}_1$ and $\tilde{h}_2$,  $\tilde{\alpha}$ and $\tilde{\beta}$,  $\phi$ and $\psi$  are still  given by or satisfy  \eqref{lambdat12}-\eqref{phipsi}.

			Second, we define the function $\Jcl(x) = \int_0^x \frac{\sqrt{P'(\sigma)}}{\sigma}\hspace{.5mm}\text{d}\sigma$, and then $\tilde{w}-\tilde{z} = 2\hspace{.5mm} \Jcl(\rho)$. It is easy to see that Assumption \ref{assumption2} implies $\Jcl$ is strictly increasing. Thus, we can write \[\rho = \Jcl^{-1}(\frac{\w-\z}{2}), \] and  the eigenvalues $\tilde{\lambda}_i$ $(i=1,2)$ can be rewritten as functions of the Riemann invariants:
\[ \tilde{\lambda}_1 = \frac{\w+\z}{2} - \sqrt{\Sigma}, \hspace{2mm} \tilde{\lambda}_2 = \frac{\w+\z}{2} + \sqrt{\Sigma},         \]
where $\Sigma=P'\left(\Jcl^{-1}\big(\frac{\tilde{w}-\tilde{z}}{2}\big) \right)$.

Next we  need  to find suitable  $\hone$ and $\htwo$ satisfying \eqref{hfunctionclas}. It follows from direct calculations that 
			\[  \tilde{\lambda}_{1\w} = \tilde{\lambda}_{2\z} =\frac{1}{2} - \frac{1}{4\sqrt{P'(\rho)}} \hspace{.5mm} P''(\rho) \hspace{.5mm} (\Jcl^{-1})'\big(\frac{\w-\z}{2}\big).      \]
			
			Notice that,   via denoting $\inf_{x\in \mathbb{R}}(w_0 - z_0)=\epsilon>0$ for some constant $\epsilon$,  we can choose\footnote{By a slight abuse of notation, whenever we write $\Sigma$ alone, we mean the expression $P'\left( \Jcl^{-1}\left( \frac{\tilde{w} - \tilde{z}}{2} \right)\right) $ but when we write $\Sigma(\sigma)$ we mean the function $\Sigma(x) = P'\left( \Jcl^{-1}\left( \frac{x}{2} \right)\right)$.}
\[    \hone= \htwo = \frac{1}{4}\ln \Sigma - \int_{\frac{\epsilon}{2}}^{\frac{\w-\z}{2}} \frac{\text{d}\sigma}{2 \sqrt{\Sigma(\sigma)}     } :=\tilde{h},  \]
			which,  along with the same argument used in the proof of Lemma \ref{lemma41}, implies that 
			 \[ \partial_- \phi= - \hspace{.5mm}\e^{-\tilde{h}} \hspace{.5mm}\lambda_{1 \z} \hspace{.5mm}\phi^2,  \hspace{2mm} \partial_+ \psi =-  \hspace{.5mm} \e^{-\tilde{h}} \hspace{.5mm} \tilde{\lambda}_{2\w} \hspace{.5mm}\psi^2.     \]Since $\tilde{\lambda}_{1\z} = \tilde{\lambda}_{2\w} >0$,  according to Assumptions \ref{assumption2}-\ref{assumption3},  there exist positive constants $\tilde{Q}_1, \tilde{Q}_2$ such that
			
			\[ \phi \leq \tilde{Q}_1, \hspace{1mm} \psi  \leq \tilde{Q}_2,  \]
			within the lifespan of the $C^1$ solution.
			
			\subsubsection{Derivation of the desired ODE inequality}	
			
			\qquad According to the same argument used in the proof of Lemma \ref{lowerboundofCE}, we therefore have,
			
			\be \label{7.5} \partial_-(\w-\z) \geq - \frac{2 \sqrt{\Sigma}}{ \sqrt[\leftroot{1}\uproot{3}4]{\Sigma}}\hspace{.5mm} \text{exp}\left( \int_{\frac{\epsilon}{2}}^{\frac{\w-\z}{2}} \frac{\text{d}\sigma}{2 \sqrt{{\Sigma(\sigma)}} } \right)  \tilde{Q}_2,       \ee 
or equivalently  due to $\partial_-(\w-\z)  = 2 \frac{\sqrt{P'(\rho)}}{\rho} \partial_-\rho$,
\be \label{7.5X}\partial_- \rho\geq -2 \hspace{.5mm} \tilde{Q}_2 \hspace{.5mm} \frac{\rho}{ \sqrt[\leftroot{1}\uproot{3}4]{P'(\rho)}    }\hspace{.5mm} \text{exp}\Big( \int_{\frac{\epsilon}{2}}^{\Jcl(\rho)} \frac{\text{d}\sigma}{2 \sqrt{\Sigma(\sigma)} } \Big) .  \ee

			\subsubsection{Lower bound estimates of the mass density.}

			\qquad Now we hope that, based on the above ODE inequality, we  can  obtain the divergence of the  integral of 
\be 
\displaystyle 
\label{7.6}\e^{-\tilde{h}}\tilde{\lambda}_{1\z}  =  \text{exp}\big( \int_{\frac{\epsilon}{2}}^{\Jcl(\rho)} \frac{\text{d}\sigma}{2 \sqrt{\Sigma(\sigma)} } \big) \hspace{.5mm}(G_1(\rho) +G_2(\rho)),  \ee  with respect to the time over $[0,+\infty)$,
where 
\be \label{GG}G_1(\rho) = \frac{1}{2} \hspace{.5mm}P'(\rho)^{-\frac{1}{4}}, \hspace{1.5mm} G_2(\rho) = \frac{1}{4} \hspace{.5mm}\rho\hspace{.5mm} P''(\rho)\hspace{.5mm} P'(\rho)^{-\frac{5}{4}}.        \ee Then similarly to the proof shown in Section 4, we can obtain the if and only if condition on the singularity formation of the classical Euler equations with general pressure law.
For simplicity, we also denote $G(\rho)=G_1(\rho)+G_2(\rho)$.

Next we   give a proper ODEs inequality for $G(\rho)$.  First, it follows from  direct calculations that 
\begin{equation}\label{g12}
\begin{split}
	G_1'(\rho) =& -\frac{1}{8} P'(\rho)^{-\frac{5}{4}} P''(\rho),\\
         G_2' (\rho) =& \frac{1}{4}\hspace{.5mm}( P''(\rho) + \rho\hspace{.5mm} P'''(\rho))\hspace{.5mm}P'(\rho)^{-\frac{5}{4}} - \frac{5}{16}\hspace{.5mm} \rho\hspace{.5mm} P''(\rho)^2\hspace{.5mm} P'(\rho)^{-\frac{9}{4}}.
\end{split}
\end{equation}
which implies that 
			\be G'(\rho) =  \text{exp}\big( \int_{\frac{\epsilon}{2}}^{\Jcl(\rho)} \frac{\text{d}\sigma}{2 \sqrt{\Sigma(\sigma)} } \big) \hspace{.5mm} \frac{\rho^2 P'(\rho) P'''(\rho) - \frac{5}{4} P''(\rho)^2 + \rho P'(\rho) P''(\rho) + P'(\rho)^2}{4 P'(\rho)^{\frac{9}{4}} \rho}. \ee Notice that Assumptions \ref{assumption2}-\ref{assumption3}  implies that 
			
			\be \label{7.11} G'(\rho) \leq  \text{exp}\big( \int_{\frac{\epsilon}{2}}^{\Jcl(\rho)} \frac{\text{d}\sigma}{2 \sqrt{\Sigma(\sigma)} } \big) \hspace{.5mm}\frac{\frac{A}{4}\big( \rho P''(\rho) + 2 P'(\rho)\big)^2}{4 P'(\rho)^{\frac{9}{4}} \rho}\ee

		Now we consider the following two cases:	
	
		\begin{itemize}	
				
\item If $  G'(\rho)> 0$, then we have 
\be \partial_- (G(\rho)) = \hspace{.5mm} G'(\rho)\hspace{.5mm} \partial_- \rho \geq - C_{\text{g}} \hspace{.5mm} G'(\rho) \hspace{.5mm} \frac{\rho}{ \sqrt[\leftroot{1}\uproot{3}4]{P'(\rho)}    }\hspace{.5mm} \text{exp}\hspace{.5mm} \big( \int_{\frac{\epsilon}{2}}^{\Jcl(\rho)} \frac{\text{d}\sigma}{2 \sqrt{\Sigma(\sigma)} } \big).  \ee 
				Moreover,  making use of \eqref{7.11}, we can further get
				
				\be \partial_- \big(G(\rho)\big)\geq - C_{\text{g}}\hspace{.5mm} \text{exp}\big( \int_{\frac{\epsilon}{2}}^{\Jcl(\rho)} \frac{\text{d}\sigma}{2 \sqrt{\Sigma(\sigma)} } \big)^2 \hspace{.5mm} 
				\frac{\big( \rho P''(\rho) + 2 P'(\rho)\big)^2}{4 P'(\rho)^{\frac{5}{2}} } = -C_{\text{g}} G(\rho)^2,\ee  for some positive constant $C_g$ independent of the time. The result follows; \vspace{3mm}
				
				\item If, on the other hand, $G'(\rho) < 0$ then simply use the upper bound on the density, $\rho_{\text{max}}$, to obtain that $G(\rho)$ is uniformly bounded below by the positive constant $G(\rho_{\text{max}})$ along the characteristic curve. Therefore, its integral along that curve diverges. The result follows.
			\end{itemize}
			
			\subsection{The general pressure law for the Relativistic Euler equations}
			\qquad In this subsection, we will give the proof for Theorem \ref{2RE}.
			\subsubsection{Notations and relations.}

			\qquad We note that the directional derivatives $\prime$ and $\backprime$, the characteristic directions $x^1$ and $x^2$, Riemann variables $w$ and $z$,  $h_1$ and $h_2$,  $\alpha$ and $\beta$,  $\xi$ and $\zeta$  are still  given by or satisfy  \eqref{lambda12}-\eqref{zwode}.	Also, $f(w,z)$, $g(w,z)$, $\tilde{f}(\rho,u)$ and $\tilde{g}(\rho,u)$ are given by Lemma \ref{fg}.		Define the helpful function \[\Jrel(x) =  \int_0^x \frac{\sqrt{P'(\sigma)}}{\sigma + \frac{P(\sigma)}{c^2}}\hspace{.5mm} d\sigma.  \]

			Rewriting $f(w,z)$ and $g(w,z)$  in terms of Riemann invariants, we denote
			\begin{gather}
			\text{ln}\left(\frac{c+\lambda_1}{c-\lambda_1}\right) = \frac{w+z}{c} + \text{ln} \left( \frac{c- \sqrt{\Lambda } }{c+ \sqrt{\Lambda }}  \right) := F(w,z), \\\text{ln}\left(\frac{c+\lambda_2}{c-\lambda_2}\right) = \frac{w+z}{c} + \text{ln} \left( \frac{c+ \sqrt{\Lambda } }{c- \sqrt{\Lambda }}  \right) := G(w,z),
			\end{gather}
			where $\Lambda =P'\left(\Jrel^{-1}\big(\frac{w-z}{2}\big) \right) $\footnote{By a slight abuse of notation, whenever we write $\Lambda$ alone, we mean the expression $P'\left( \Jrel^{-1}\left( \frac{\tilde{w} - \tilde{z}}{2} \right)\right) $ but when we write $\Lambda(\sigma)$ we mean the function $\Sigma(x) = P'\left( \Jrel^{-1}\left( \frac{x}{2} \right)\right)$.}.
			It follows from the direct calculation that 
			 \begin{equation}\label{summary}
			 \begin{split}
			\lambda_1(w,z) =& c \left( 1- \frac{2}{\e^{F(w,z)} +1}    \right), \ \  \hspace{.5mm} \lambda_2(w,z) = c \left( 1- \frac{2}{\e^{G(w,z)} +1}    \right),\\
			\lambda_2 - \lambda_1 =& 2\hspace{.5mm} c\hspace{.5mm} \frac{\e^G - \e^F}{\left(\e^F +1\right)\left(\e^G+1\right)} \\
				=&2 \hspace{.5mm} c \hspace{.5mm} \e^{\frac{w+z}{c}} \hspace{.5mm} \frac{        \frac{4 \hspace{.5mm} c \hspace{.5mm} \sqrt{\Lambda } }{c^2 - \Lambda  }    }{       \left( 1+ \e^{\frac{w+z}{c}}\hspace{.5mm} \frac{ c-\sqrt{\Lambda} }{c+\sqrt{\Lambda}     } \right)    \left(1+ \e^{\frac{w+z}{c}}\hspace{.5mm} \frac{ c+\sqrt{\Lambda} }{c-\sqrt{\Lambda} } \right)       },\\
				\lambda_{1w} = &\frac{2\hspace{.5mm}c\hspace{.5mm} \e^{F(w,z)}}{\left( \e^{F(w,z)}+1\right)^2}\hspace{.5mm} {F}_w (w,z),\quad \e^{F(w,z)} = \e^{\frac{w+z}{c}} \frac{c- \sqrt{\Lambda} }{c+ \sqrt{\Lambda}}, \\
				 \end{split} 
				 \end{equation}
				 and 
				 \begin{equation}\label{summary1}						
				 F_w(w,z)  =     \frac{1}{c} - \frac{2c}{c^2 - \Lambda}\cdot \frac{P''\left(\Jrel^{-1}\big(\frac{w-z}{2}\big) \right) \left(\Jrel^{-1}\right)^{'} \left( \frac{w-z}{2}\right) }{4 \hspace{.5mm} \sqrt{\Lambda}}, 
				 \end{equation}

which, imply that 
			\be \label{l1w1} \frac{\lambda_{1w}}{\lambda_1 - \lambda_2} = \frac{\frac{2\hspace{.5mm}c\hspace{.5mm} \e^{F}}{\left( \e^{F}+1\right)^2}\hspace{.5mm} F_w}{-2\hspace{.5mm} c\hspace{.5mm} \frac{\e^G - \e^F}{\left(\e^F +1\right)\left(\e^G+1\right)}} = - \frac{\e^G+1}{\left( \e^F+1\right) \left( \e^G - \e^F \right)}( \e^F )_w = -\left( \frac{1}{\e^F+1}+ \frac{1}{\e^G- \e^F}\right)(\e^F)_w. \ee  
			
		Notice that
\[ \text{ln}\left(\e^{G+F} -\e^{2F} \right)_w = \frac{ 1}{\e^{G+F}-\e^{2F}}\left( \frac{2}{c} \e^{2 \frac{w+z}{c}} - 2\e^F (\e^F)_w \right)  = - \frac{2 (\e^F)_w}{\e^G-\e^F} + \frac{2 \e^{\frac{2(w+z)}{c}}}{c \left(\e^{G+F} - \e^{2F}\right) },   \] 
then one can obtain that 
\[   -\frac{1}{\e^G-\e^F}(\e^F)_w = \frac{1}{2} \text{ln}(\e^{G+F}- \e^{2F})_w - \frac{\e^\frac{2(w+z)}{c}}{c(\e^{\frac{2(w+z)}{c}}-\e^{2F})} .                     \] Then, together with \eqref{l1w1}, one gets
			\be  \frac{\lambda_{1w}}{\lambda_1 - \lambda_2} = -\left( \text{ln}\left( \e^F+1
			\right) \right)_w + \frac{1}{2} \left( \text{ln}\left( \e^{G+F}-\e^{2F}
			\right) \right)_w  - \frac{1}{c(1-\e^{F-G})}.   \ee 
Therefore, we can choose
\be  h_1 = -\text{ln}\left(\e^F +1 \right) + \frac{1}{2} \text{ln}\left(\e^{G+F} - \e^{2F}\right) -   \int_{\epsilon/2}^{\frac{w-z}{2}} \frac{\left(c +\sqrt{\Lambda(\sigma)}\right)^2}{2 \hspace{.5mm}c^2 \hspace{.5mm} \sqrt{\Lambda(\sigma)}}\text{d}\sigma.         \ee Similarly, one can obtain 
\be \lambda_{2z} = \frac{2\hspace{.5mm}c\hspace{.5mm} \hspace{.5mm} \e^G}{\left(\e^G+1\right)^2}\hspace{.5mm}G_z.\ee Then one has 
\be \begin{split}\frac{\lambda_{2z}}{\lambda_2 - \lambda_1} =& \frac{\frac{2\hspace{.5mm}c\hspace{.5mm} \hspace{.5mm} \e^G}{\left(\e^G+1\right)^2}\hspace{.5mm}G_z}{2\hspace{.5mm} c\hspace{.5mm} \frac{\e^G - \e^F}{\left(\e^F +1\right)\left(\e^G+1\right)} }= \frac{\e^F+1}{(\e^G+1)(\e^G -\e^F)}(\e^G)_z \\
=& \left( - \frac{1}{\e^G+1} + \frac{1}{\e^G- \e^F}\right)\left(\e^G\right)_z\\
=&- \left(  \text{ln} \left(\e^G+1\right)  \right)_z + \frac{1}{\e^G-\e^F}\left( \e^G \right)_z.  \end{split} \ee

It follows fromt the direct calculation that 
\be \frac{1}{2}\hspace{.5mm} \text{ln}\left( \e^{2G}- \e^{G+F}        \right)_z = \frac{\left( \e^G\right)_z}{\e^G - \e^F}- \frac{1}{c} \frac{\e^{G+F}}{\e^{2G}-\e^{G+F}}.  \ee 
Then 
\be \frac{1}{\e^G-\e^F} \big(\e^G \big)_z = \frac{1}{2} \hspace{.5mm} \text{ln} \left( \e^{2G} - \e^{G+F} \right)_z + \hspace{.5mm} \frac{1}{c (\e^{G-F}-1)}. \ee 
We can therefore choose 
\be\label{7.33}  h_2 = -\text{ln}\left(\e^G+1\right)+ \frac{1}{2}\hspace{.5mm} \text{ln}\left( \e^{2G} - \e^{G+F} \right)  - \int_{\frac{\epsilon}{2}}^{\frac{w-z}{2}}  \frac{\left(c - \sqrt{\Lambda(\sigma)}\right)^2}{2 \hspace{.5mm} c^2 \hspace{.5mm} \sqrt{\Lambda(\sigma)}}\text{d}s.     \ee

Set, for convenience,
			
			\be \mathcal{I}_1= \int_{\frac{\epsilon}{2}}^{\frac{w-z}{2}}  \frac{\left(c + \sqrt{\Lambda(\sigma)}\right)^2}{2 \hspace{.5mm} c^2 \hspace{.5mm} \sqrt{\Lambda(\sigma)}}\text{d}\sigma , \hspace{2mm}\mathcal{I}_2 =  \int_{\frac{\epsilon}{2}}^{\frac{w-z}{2}}  \frac{\left(c - \sqrt{\Lambda(\sigma)}\right)^2}{2 \hspace{.5mm} c^2 \hspace{.5mm} \sqrt{\Lambda(\sigma)}}\text{d}\sigma.  \ee
			
			According to  \eqref{summary} and \eqref{7.33}, one can obtain 
\be\label{keystep}\begin{split} 			\frac{\lambda_2 - \lambda_1}{\e^{h_2}} =& C_{\text{g}} \hspace{.5mm} \frac{\e^G- \e^F}{\left(\e^G+1\right) \left( \e^F +1\right)} \cdot \frac{\e^G+1}{\sqrt{\e^{2G} - \e^{G+F}}} \cdot \exp(\mathcal{I}_2)\\
				=& C_{\text{g}} \frac{\e^{\frac{w+z}{c}} \frac{4 \hspace{.5mm} c \hspace{.5mm} \sqrt{P'(\rho)} }{c^2-P'(\rho)}}{1+ \e^{\frac{w+z}{c}}\frac{c-\sqrt{P'(\rho)}}{c+\sqrt{P'(\rho)}}}\cdot \frac{c-\sqrt{P'(\rho)}}{\e^{\frac{w+z}{c}}  \sqrt[\leftroot{1}\uproot{3}4]{P'(\rho )}}\cdot \text{exp}(\mathcal{I}_2) \\
				=& C_{\text{g}} \frac{\sqrt[\leftroot{1}\uproot{3}4]{P'(\rho )}\hspace{1mm}\text{exp}(\mathcal{I}_2)}{\left(1+\frac{\sqrt{P'(\rho)}}{c}\right)\left( 1+ \e^{\frac{w+z}{c}}\frac{c-\sqrt{P'(\rho)}}{c+\sqrt{P'(\rho)}} \right)},
			\end{split} \ee 
for some positive constant $C_g$ independent of the time.			
			
			Furthermore,
\be\begin{split} 
				\e^{-h_1}\lambda_{1z} =& \frac{\e^F+1}{\sqrt{\e^{G+F}-\e^{2F}}}\text{exp}(\mathcal{I}_1)  \left( \frac{1}{c} + \frac{2cP''(\rho) \left( \rho + \frac{P(\rho)}{c^2}\right)}{4P'(\rho)(c^2 - P'(\rho))}   \right)   \\ 
				=& \frac{1 + \e^{\frac{w+z}{c}} \frac{c- \sqrt{P'(\rho)}}{c+\sqrt{P'(\rho)}}   }{\sqrt{\e^{\frac{2(w+z)}{c}}} \sqrt{\frac{4 c \sqrt{P'(\rho)}}{\left(c+\sqrt{P'(\rho)}\right)^2}}}\text{exp}(\mathcal{I}_1)  \left( \frac{1}{c} + \frac{2cP''(\rho) \left( \rho + \frac{P(\rho)}{c^2}\right)}{4P'(\rho)(c^2 - P'(\rho))}   \right)  \\ 
				=& C_{\text{g}} \frac{\left(c+\sqrt{P'(\rho)}\right) \left( 1 + \e^{\frac{w+z}{c}} \frac{c-\sqrt{P'(\rho)}}{c+ \sqrt{P'(\rho)}}      \right) }{\e^{\frac{w+z}{c}} \sqrt[\leftroot{1}\uproot{3}4]{P'(\rho )}} \text{exp}(\mathcal{I}_1)  \left( \frac{1}{c} + \frac{2cP''(\rho) \left( \rho + \frac{P(\rho)}{c^2}\right)}{4P'(\rho)(c^2 - P'(\rho))}   \right) \\
				:= &H(\rho).
			\end{split} \ee Define
			
			\be H_1 = \frac{1}{\sqrt[\leftroot{1}\uproot{3}4]{P'(\rho )}},\hspace{2mm}  \quad H_2 = \frac{P''(\rho)\left(\rho + \frac{P(\rho)}{c^2}\right)}{2\left(1- \frac{P'(\rho)}{c^2}\right)P'(\rho)^{\frac{5}{4}}.}  \ee
	then one has 
			
			\[ \e^{-h_1}\lambda_{1z}= C_{\text{g}} \frac{\left(1+\frac{\sqrt{P'(\rho)}}{c}\right) \left( 1 + \e^{\frac{w+z}{c}} \frac{c-\sqrt{P'(\rho)}}{c+ \sqrt{P'(\rho)}}      \right) }{\e^{\frac{w+z}{c}} }\text{exp}(\mathcal{I}_1)  \left( H_1(\rho)+H_2(\rho) \right).      \]
			
	It is easy to check that the conclusions of Lemmas \ref{lemma5.1}-\ref{5.4} and \ref{Q1Q2lemma} still hold for the current  pressure law.		
			
			\subsubsection{Derivation of the desired ODE inequality}

			\qquad First, it is obvious that 
\begin{equation}\label{uplow}
\begin{split}
C^{-1}_g\leq 1+ \frac{\sqrt{P'(\rho)}}{c}\leq C_g,\ \ 
 C^{-1}_g\leq 1 + \e^{\frac{w+z}{c}} \frac{c-\sqrt{P'(\rho)}}{c+ \sqrt{P'(\rho)}} \leq C_g,\ \ C^{-1}_g\leq \e^{\frac{w+z}{c}}\leq C_g,
\end{split}\end{equation}
for some positive constant $C_g$ independent of the time.	

Given two functions $f_1, f_2$,  we denote  $f_1 \approx f_2$ if and only if there are positive constants $C_1$ and $ C_2$ such that
\[ C_1 f_1 \leq f_2 \leq C_2 f_1, \hspace{2mm} \text{pointwisely}.  \]
With this in mind, the above three remarks allow us to conclude that
			
			\begin{equation}
			\e^{-h_1} \lambda_{1z} \approx \exp(\mathcal{I}_1)\left(H_1(\rho)+H_2(\rho)\right).
			\end{equation} 
			
			Second, we observe that 
			\begin{lemma}\label{lemma7.1}
				\[  H_1(\rho)+H_2(\rho) \approx G_1(\rho) + G_2(\rho),             \]where $G_1$ and $G_2$  are given by \eqref{GG}.
			\end{lemma} \begin{proof}
				First,  $H_1 \equiv 2 G_1 \geq 0$. Second, 
				$$\frac{2(1+\gamma^2)}{1-\gamma^2}\hspace{.5mm}G_2(\rho)\hspace{.5mm}\geq \hspace{.5mm}H_2(\rho) = 2\hspace{.5mm}G_2(\rho) \hspace{.5mm}\frac{1+\frac{P(\rho)}{c^2 \rho}}{1- \frac{P'(\rho)}{c^2}} \geq 2\hspace{.5mm}G_2(\rho),$$ where one has  used the facts that  $P(0)=0$ and $P'(\rho) \leq c^2\gamma^2$.  Then, one can obtain that 
\[ \frac{1+\gamma^2}{1-\gamma^2}(G_1+G_2) \geq  H_1 +H_2 \geq 2(G_1+G_2), \hspace{2mm} \text{pointwisely}.              \] 
			\end{proof}
			
			 Next, we will show that
	\begin{proposition}
			$\exp(\mathcal{I}_1) \geq C_{\text{g}} \hspace{.5mm} \exp\left( \int_{\frac{\epsilon}{2}}^{J_{\text{clas}}(\rho)} \frac{\text{d}\sigma}{2 \sqrt{\Sigma(\sigma)} } \right)$.
			\end{proposition}
			
			\begin{proof}  First, notice that,  from Assumption   \ref{assumption2}, one has		 
		\be\label{pcon} P(\rho) \leq \rho\hspace{.5mm} P'(\rho).\ee		
	
	Second, set $x = J_{\text{rel}}^{-1}(s)$ and substitute the integral variable, then one can rewrite $\exp(\mathcal{I}_1) $ as 
	\[  \exp(\mathcal{I}_1) =   \exp\left(\int_{J_{\text{rel}}^{-1}(\epsilon/2)}^\rho  \Psi(x)  \text{d}x       \right)   \quad \text{with}  \quad  \Psi(x)= \frac{\left(1 + \frac{\sqrt{P'(x)}}{c} \right)^2}{2\left(x + \frac{P(x)}{c^2}\right)}.   \]
	
	Similarly, one has
\[   \exp\left( \int_{\frac{\epsilon}{2}}^{J_{\text{clas}}(\rho)} \frac{\text{d}\sigma}{2 \sqrt{\Lambda(\sigma)} } \right) =   \exp\left(\int_{J_{\text{clas}}^{-1}(\epsilon/2)}^\rho  \frac{1}{2x }   \text{d}x       \right).        \]
				
It follows from \eqref{pcon} that 			
				\[   \left(1 + \frac{\sqrt{P'(x)}}{c} \right)^2\geq \frac{1}{x}\left(x + \frac{P(x)}{c^2}\right).    \]
Then, combining  the above relations, one can obtain that 
				\be\begin{split}  \exp(\mathcal{I}_1) =&   \exp\left(\int_{J_{\text{rel}}^{-1}(\epsilon/2)}^\rho  \Psi(x)   \text{d}x       \right) 
					=  C_{\text{g}}   \exp\left(\int_{J_{\text{clas}}^{-1}(\epsilon/2)}^\rho \Psi(x)  \text{d}x       \right)\geq   C_{\text{g}}  \exp\left(\int_{J_{\text{clas}}^{-1}(\epsilon/2)}^\rho  \frac{1}{2x }   \text{d}x       \right). \end{split} \ee
			\end{proof}\par \noindent This brings us back to the classical expression of Subsection \ref{subsection7.1}.
			\vspace{3mm}
			
	\subsubsection{Lower bound estimate on the mass-energy density} 
	\qquad According to the conclusions obtained in the above two steps, now we give the desired lower bound estimates. First, according to the definitions of $w$ and $z$,  \eqref{startingpoint} and \eqref{keystep}, one can obtain that 
\be \begin{split}\label{7.41} \rho^\prime \geq &- C_{\text{g}}\hspace{.5mm} \exp(\mathcal{I}_2)\hspace{.5mm} \frac{\rho+ \frac{P(\rho)}{c^2}}{P'(\rho)^{1/4}} \geq - C_{\text{g}}\hspace{.5mm} \exp(\mathcal{I}_2)\hspace{.5mm} \frac{\rho}{P'(\rho)^{1/4}} \\
				\geq & - C_{\text{g}}\hspace{.5mm} \exp\left(\int_{J_{\text{clas}}^{-1}(\epsilon/2)}^\rho  \frac{1}{2x }   \text{d}x       \right) \frac{\rho}{P'(\rho)^{1/4}},\end{split} \ee
		where one has  used the fact that $P(\rho) \leq c^2 \hspace{.5mm}  \hspace{.5mm}\rho.$ 
		
		It follows from   \eqref{7.41} that  the integral of the classical expression
\[  \exp\left(\int_{J_{\text{clas}}^{-1}(\epsilon/2)}^\rho  \frac{1}{2x }   \text{d}x       \right) (G_1(\rho) + G_2(\rho) )    \] diverges according to the results obtained in  Subsection \ref{subsection7.1}. Since
			
			\[   \exp(\mathcal{I}_1)(H_1(\rho) +H_2(\rho)) \geq C_{\text{g}} \exp\left(\int_{J_{\text{clas}}^{-1}(\epsilon/2)}^\rho  \frac{1}{2x }   \text{d}x       \right) (G_1(\rho) + G_2(\rho) ),      \] the result follows from the arguments used in Section 7.1.3.

			\section{Appendix}\label{section8}
			
		\qquad In order to support our theory on the singularity formation shown in Section 6, in this appendix, we show the     local-in-time well-posedness of the smooth solution   for the Relativistic Euler equations \eqref{RE} with initial vacuum in multi-dimensional spacetime. The detailed proof can be found in Lefloch-Ukai	 \cite{LeflochUkai}.

			\subsection{Symmetric formulation of \eqref{RE} allowing vacuum}
			
		\qquad First for simplicity, in the following we denote $c_s=\sqrt{P'(\rho)}$ as the sound speed in the fluid.  We refer to
			\begin{equation}
			\label{E:3.3}z_\pm:=S(\rho)\pm R(u)\end{equation} 
			as the generalized Riemann invariant variables,  where \begin{gather}\label{rrss1}
	R(u) =: \frac{c}{2}\hspace{.5mm} \ln\Big(\frac{c+|u|}{c-|u|}\Big),\quad S(\rho)=:\int_0^\rho \frac{\sqrt{P'(\sigma)}}{\sigma + \frac{P(\sigma)}{c^2}} \text{d}\sigma.
	\end{gather}

We also introduce the projection operator and the normalized velocity 
			\begin{equation}
			\label{E:3.4}
			E(u):=\mathbb{I}_3-\tilde{u}\otimes\tilde{u},\quad \quad \tilde{u}=\frac{u}{|u|},		
			\end{equation} where $\mathbb{I}_3$ represents a $3\times 3$ unit matrix. Then one has

			\begin{lemma}
				\label{L:3.1}\text{\cite{LeflochUkai}} In terms of the generalized Riemann invariant variables
				$(z_+,z_-)$ and the normalized velocity $\tilde{u}$ defined in
				\eqref{E:3.3}-\eqref{E:3.4}, the relativistic Euler equations \eqref{RE} take
				the following symmetric form
				\begin{equation}\label{E:2.50}
A_0(W)W_t+\sum_{j=1}^3 A_jW_{x_j}=0,
\end{equation}				
with $W=(z_+,z_-,\tilde{u})$, in which $A^0(W)$ and $A^j(W)$ are
			\begin{equation}
			\label{E:3.23} A^0(W)=\left(
			\begin{array}{ccc}
			a_0&0&0\\
			0&b_0&0\\
			0&0&c_0|u|^2\mathbb{I}_3\\
			\end{array}\right),\quad A^j(W)=\left(
			\begin{array}{ccc}
			a_1\tilde{u}_j&0&a_2|u|e_j\\
			0&b_2\tilde{u}_j&-a_2 |u|e_j\\
			a_2ve^\top_j&-a_2|u|e^\top_j&c_0|u|^2u_j\mathbb{I}_3\\
			\end{array}\right),
			\end{equation}
			and
			\begin{equation}\label{E:3.23a}
			\begin{split}a_0=&1+{|u|c_s}/{c^2},\quad b_0=1-{|u|c_s}/{c^2},\quad
			c_0=\frac{2}{1-{|u|^2}/{c^2}},\\[2pt]
			a_1=&\frac{1-{c_s^2}/{c^2}}{1-{c_s|u|}/{c^2}}(|u|+c_s),\quad b_1=\frac{1-{c_s^2}/{c^2}}{1+{c_s|u|}/{c^2}}(|u|-c_s),\\[2pt]
			a_2=&c_s,\quad e_j=(E_{j1}(u), E_{j2}(u), E_{j3}(u)),\quad j=1, 2, 3.
			\end{split}
			\end{equation}			\end{lemma}

			We consider  the Cauchy problem associated with \eqref{E:2.50} where initial data  given on the initial hyperplane 
			$$
			\mathcal{H}_0: \quad t=0.
			$$

			First, we observe that
			\begin{equation}\label{E:3.27}
			<A^0(W)\eta,\eta>=a_0|\eta_1|^2+b_0|\eta_2|^2+c_0|u|^2|\hat{\eta}|^2,
			\end{equation}
			where $<\cdot,\cdot>$ denotes the Eucilidian inner product in $\mathbb{R}^{5}$ and
			$$
			\eta=(\eta_1,\eta_2,\ldots,\eta_5)=(\eta_1,\eta_2,\hat{\eta})\in \mathbb{R}^5,\quad \hat{\eta}=(\eta_3,\eta_4,\eta_5)\in \mathbb{R}^3.
			$$			
			From \eqref{E:3.23a}, the matrix $A^0(W)$ can be positive definite only if the velocity $u$
			never vanishes. According to Friedlichs-Lax-Kato theory \cite{Kato, Majda}, a local in time solution exists and is unique in the Sobolev space $H^\sigma$ for $\sigma>\frac{5}{2}$. However, this lower bound on velocity is not physically realistic.

			In order to deal with this difficulty, 
			Lefloch and Ukai \cite{LeflochUkai} apply a well-chosen Lorentz transformation, which allows for the fact that  the Lorentz-transformed velocity  does not
			exceed some threshold and remains bounded away from the light speed.
			\subsection{Lorentz transformation}\label{S:3.2}
		\qquad The Lorentz transformation $(t,x)\rightarrow (\bar{t},\bar{x})$ associated with the vector $U\neq 0$ is defined by   
			\begin{equation}\label{E:3.23}\begin{cases}
			\bar{t}=\varpi\left(t-\frac{U\cdot x}{c^2}\right),\\
			\bar{{\bf x}}=-\varpi U t+\left(\mathbb{I}_3+(\varpi-1)\frac{U\otimes U}{U^2}\right)x,\\
			\bar{\rho}(\bar{t},\bar{x})=\rho(t,x),\\
			\bar{u}=\frac{d\bar{x}}{d\bar{t}}=\frac{1}{1-U\cdot u/c^2}\left(-U+\left(\varpi \mathbb{I}_3+(1-\varpi^{-1} )\frac{U\otimes U}{U^2}\right)u\right),			
			\end{cases}
			\end{equation}
			where $\varpi=\frac{1}{\sqrt{1-U^2/c^2}}$ represents the Lorentz factor.

			Using the Lorentz invariance property of the Euler equations, \eqref{E:2.50} can be also expressed in the transformed
			coordinates $(\bar{t},\bar{x})$ defined by \eqref{E:3.23}, that is
			\begin{equation} \label{E:3.28}
			A^0(\bar{W})(\bar{W})_{\bar{t}}+\sum_{j=1}^3A^j(\bar{W})_{\bar{x}_j}=0,
			\end{equation}
			where $\bar{W}=(\bar{z}_+,\bar{z}_-,\bar{\tilde{u}})$ is defined from the transformed unknowns $(\bar{\rho},\bar{u})$.
			The expression of \eqref{E:3.27} becomes
			\begin{equation}\label{E:3.29}
			<A^0(\bar{W})\eta,\eta>=\bar{a}_0|\eta_1|^2+\bar{b}_0|\eta_2|^2+\bar{c}_0|\bar{u}|^2|\hat{\eta}|^2,
			\end{equation}
			where $\bar{a}_0$, $\bar{b}_0$ and $\bar{c}_0$ are defined by \eqref{E:3.23a} with $(\rho,u)$ replaced by $(\bar{\rho},\bar{u})$.
			
			In view the upper and lower bounds \eqref{E:3.29}, we conclude the transformed matrix $A^0(\bar{W})$
			is positive definite in the coordinate system $(\bar{t},\bar{x})$, Hence Friedlichs-Lax-Kato theory \cite{Kato, Majda}
			applies to the initial value problem for \eqref{E:3.28}, provided initial data are imposed on the initial hypersurface
			$\bar{t}=0$. In the relativistic setting, the initial plane $H: t=0$ is not preserved by the transformation \eqref{E:3.23}.
			However, in the new coordinate system $(\bar{t},\bar{x})$ the initial plane becomes
			$$
			\bar{H}_0:\quad \quad \bar{t}=-\frac{U\cdot \bar{x}}{c^2}.
			$$
			
			In order to prove local well-posedness for the oblique initial-value problem \eqref{E:3.28} with data in $\bar{H}_0$, it is convenient
			to introduce a further change of coordinates
			\begin{equation}\label{E:3.30}
		t'=\bar{t}+\frac{{U\cdot \bar{x}}}{c^2},\quad x'=\bar{x},
			\end{equation}
			which maps the hyperplane $\bar{H}_0$ to the hyperplane
			$$
			H'_0:\quad \quad t'=0.
			$$
			
			This transformation puts the system \eqref{E:3.29} into the form
			\begin{equation} \label{E:3.31}
			B^0(W')(W')_{t'}+\sum_{j=1}^3B^j(W')
			W'_{x^{'}_j}=0,
			\end{equation}
			where $W'=(z'_+, z'_-, \tilde{u}')^\top$ is defined from the transformed unknowns $(\rho'(t',x')=\bar{\rho}(\bar{t},\bar{x}),u'(t',x')=\bar{u}(\bar{t},\bar{x}))$,
			$$
			B^0(W')=A^0(\bar{W})+\frac{1}{c^2} \sum_{j=1}^d U_j A_j(\bar{W}), \quad B_j(W')=A_j(\bar{W}).
			$$
 It has been proved in \cite{LeflochUkai} that $B^0(W')$ is positive definite via choosing proper $U$. Here we omit its proof.

			\subsection{Desired local-in-time well-posedenss}	
			\quad In turn, Friedrichs-Lax-Kato theory \cite{Kato, Majda} guarantees the existence of a solution defined
			in a small neighborhood of this hyperplane $H'_0$. Making the transformation back to the original variables,
			we obtain a solution in a small neighborhood of the initial line $t=0$. This finally gives the following theorem: 
			\begin{theorem}\cite{LeflochUkai}\label{leflochukai}
			If the initial data $( \rho_0, u_0)$ satisfy the following regularity conditions:
\begin{equation}\label{th78}
\begin{split}
0\leq \rho_0\leq M_1,\quad  |u_0|<c,\quad \sqrt{P'(\rho_0)}<c,\quad \left( \rho_0, u_0\right)\in H^3_{ul}(\mathbb{R}^3),
\end{split}
\end{equation}
for some constants $M_1>0$,
then there exists a positive  time $T_*$ and a unique classical  solution $(\rho, u)(t,x)$ in $[0, T_*]\times \mathbb{R}^3$ to   the  Cauchy problem  (\ref{RE})-(\ref{P}) with \eqref{initial1} satisfying 
\begin{equation}\label{th77}
\begin{split}
&\rho\geq 0,\quad  |u|<c,\quad \sqrt{P'(\rho)}<c,\\
& (\rho,u)\in C([0,T_*];H^3_{ul}(\mathbb{R}^3))\cap C^1([0,T_*];H^2_{ul}(\mathbb{R}^3)),\\
&  u_t+u\cdot\nabla u =0\quad  \text{whenever} \quad  \rho(t,x)=0.
\end{split}
\end{equation}
\end{theorem}


			\bigskip

			\textbf{Acknowledgements} The authors wish to thank Profs. Gui-Qiang G. Chen, Yongcai Geng and Qian Wang for many useful discussions. Nikolaos Athanasiou is supported by the Engineering and Physical Sciences Research Council grant [EP/L015811/1].  Shengguo Zhu is partially  supported 
			by  Australian Research Council grant DP170100630 and  the Royal Society-- Newton International Fellowships NF170015.

		\end{document}